\documentclass[a4paper]{amsart}
\usepackage{amssymb, enumitem}
\usepackage[all]{xy}
\usepackage{hyperref, aliascnt, graphicx, scalerel}

\usepackage[capitalise, nameinlink, noabbrev, nosort]{cleveref} 

\theoremstyle{plain}
\newtheorem{lma}{Lemma}[section]
\crefname{lma}{Lemma}{Lemmata}
\newtheorem{thm}[lma]{Theorem}
\crefname{thm}{Theorem}{Theorems}
\newtheorem{cor}[lma]{Corollary}
\crefname{cor}{Corollary}{Corollaries}
\newtheorem{prp}[lma]{Proposition}
\crefname{prp}{Proposition}{Propositions}

\theoremstyle{definition}
\newtheorem{pgr}[lma]{}
\crefname{pgr}{Paragraph}{Paragraphs}
\newtheorem{dfn}[lma]{Definition}
\crefname{dfn}{Definition}{Definitions}

\theoremstyle{remark}
\newtheorem{rmk}[lma]{Remark}
\crefname{rmk}{Remark}{Remarks}
\newtheorem{exa}[lma]{Example}
\crefname{exa}{Example}{Examples}
\newtheorem{qst}[lma]{Question}
\crefname{qst}{Question}{Questions}
\newtheorem{ntn}[lma]{Notation}
\crefname{ntn}{Notation}{Notations}

\newcounter{theoremintro}

\theoremstyle{plain}
\newtheorem{thmIntro}[theoremintro]{Theorem}
\crefname{thmIntro}{Theorem}{Theorems}
\newtheorem{prpIntro}[theoremintro]{Proposition}
\crefname{prpIntro}{Proposition}{Propositions}
\newtheorem{dfnIntro}[theoremintro]{Definition}

\def\today{\number\day\space\ifcase\month\or   January\or February\or
   March\or April\or May\or June\or   July\or August\or September\or
   October\or November\or December\fi\   \number\year}

\newcommand{\NN}{{\mathbb{N}}}
\newcommand{\KK}{{\mathbb{K}}}
\newcommand{\RR}{{\mathbb{R}}}
\newcommand{\Bdd}{{\mathcal{B}}}
\newcommand{\spec}{{\mathrm{sp}}}
\newcommand{\supp}{{\mathrm{supp}}}
\newcommand{\ca}{$\mathrm{C}^*$-algebra}
\newcommand{\CatCu}{\ensuremath{\mathrm{Cu}}}

\newcommand{\axiomO}[1]{(O#1)}
\newcommand{\OP}{\mathcal{O}}

\newcommand{\stHom}{${}^*$-homomorphism}
\newcommand{\soft}{{\rm{soft}}}

\newcommand{\andSep}{\,\,\,\text{ and }\,\,\,}
\newcommand{\CuSgp}{$\CatCu$-semi\-group}
\newcommand{\CuMor}{$\CatCu$-morphism}

\DeclareMathOperator{\QT}{QT}
\DeclareMathOperator{\Cu}{Cu}
\DeclareMathOperator{\linSpan}{span}
\DeclareMathOperator{\Ann}{Ann}

\newcommand\wh[1]{\hstretch{2}{\widehat{\hstretch{.5}{#1}}}}

\title{Soft operators in C*-algebras}
\date{\today}

\author{Hannes Thiel}
\address{Hannes~Thiel, 
Department of Mathematical Sciences, Chalmers University of Technology and University of
Gothenburg, Gothenburg SE-412 96, Sweden.}
\email{hannes.thiel@chalmers.se}
\urladdr{www.hannesthiel.org}

\author{Eduard Vilalta}
\address{Eduard Vilalta,
Departament de Matem\`{a}tiques,
Universitat Aut\`{o}noma de Barcelona,
08193 Bellaterra, Barcelona, Spain}
\email{eduard.vilalta@uab.cat}
\urladdr{www.eduardvilalta.com}

\thanks{
The first named author was partially supported by the Knut and Alice Wallenberg Foundation (KAW 2021.0140).
The second named author was partially supported by MINECO (grant No.\ PID2020-113047GB-I00 and No.\ PRE2018-083419), and by the Departament de Recerca i Universitats de la Generalitat de Catalunya (grant No.\ 2021-SGR-01015).
}

\begin{document}

\begin{abstract}
We say that a \ca{} is soft if it has no nonzero unital quotients, and we connect this property to the Hjelmborg-R{\o}rdam condition for stability and to property (S) of Ortega-Perera-R{\o}rdam.
We further say that an operator in a \ca{} is soft if its associated hereditary subalgebra is, and we provide useful spectral characterizations of this concept.

Of particular interest are \ca{s} that have an abundance of soft elements in the sense that every hereditary subalgebra contains an almost full soft element.
We show that this property is implied by the Global Glimm Property, and that every \ca{} with an abundance of soft elements is nowhere scattered. 
This sheds new light on the long-standing Global Glimm Problem of whether every nowhere scattered \ca{} has the Global Glimm Property.
\end{abstract}

\maketitle

\section{Introduction}

A recurring property in the study of non-simple \ca{s} that ensures sufficient noncommutativity is the notion of having no nonzero elementary ideal-quotients; 
see, for example, \cite{EllRor06Perturb, RobTik17NucDimNonSimple, AntPerRobThi22CuntzSR1}. 
Given the stark contrast between such \ca{s} and those that are scattered (as defined by Jensen in \cite{Jen77ScatteredCAlg}), the property was termed \emph{nowhere scatteredness} in \cite{ThiVil21arX:NowhereScattered}.

Despite this notion being the natural substitute for simplicity in a variety of scenarios, one oftens finds oneself working with the more technical condition of having the \emph{Global Glimm Property}. 
This property, introduced by Kirchberg and R{\o}rdam in \cite{KirRor02InfNonSimpleCalgAbsOInfty}, implies nowhere scatteredness and, in fact, it is known that both notions agree for a wide class of \ca{s}. 
However, whether these two conditions are the same in general remains an open problem, known as the Global Glimm Problem; 
see \cite[Section~1]{ThiVil22arX:Glimm} for an overview on the subject. 
In essence, and as in many other problems in operator algebras, the Global Glimm Problem asks if a natural condition and its more useful (albeit more technical) strengthening coincide.

A new approach to attacking this problem was developed in \cite{ThiVil22arX:Glimm}, where the statement of the Global Glimm Problem was reformulated in terms of the Cuntz semigroup, a powerful invariant for \ca{s} introduced by Cuntz in \cite{Cun76ContinuitySeminorms} and later studied more abstractly by Coward, Elliott, and Ivanescu in \cite{CowEllIva08CuInv}. The theory of Cuntz semigroups, which was further developed in \cite{AntPerThi18TensorProdCu, AntPerThi20AbsBivariantCu, AntPerThi20AbsBivarII, AntPerThi20CuntzUltraproducts, AntPerRobThi21Edwards}, has seen many important applications in the structure and classification of \ca{s};
see, for example, \cite{Tom08ClassificationNuclear, Rob12LimitsNCCW, Win12NuclDimZstable, RobTik17NucDimNonSimple, Thi20RksOps, AntPerRobThi22CuntzSR1}.
For a recent introduction to the theory we refer to \cite{GarPer23arX:ModernCu}.

In order to exploit the approach from \cite{ThiVil22arX:Glimm}, it becomes imperative to deepen our understanding of the Global Glimm Property. In this paper, we do so by characterizing the property in terms of \emph{soft elements}.

\begin{dfnIntro}[\ref{dfn:SoftElement}]
Given a \ca{} $A$, we say that a positive element $a\in A_+$ is \emph{soft} if no nonzero quotient of $\overline{aAa}$ is unital.
\end{dfnIntro}

This definition is in analogy to the notion of having a soft Cuntz class, as introduced in \cite{AntPerThi18TensorProdCu}, and draws inspiration from the characterization of stability for purely infinite \ca{s} obtained in \cite[Theorem~4.24]{KirRor00PureInf}: 
A $\sigma$-unital, purely infinite \ca{} is stable if and only if none of its nonzero quotients is unital. 
In fact, despite never having a name, the notion of softness appears throughout the literature. 
For example, one could restate part of \cite[Theorem~B]{BosGabSimWhi22NucDimOStableMaps} by saying that a separable, nuclear $\mathcal{O}_\infty$-stable \ca{} has finite decomposition rank if and only if every positive element in $A$ is soft.

We begin our investigation by first developing the basic theory of softness for \ca{s}, where we say that a \ca{} is \emph{soft} if it has no nonzero, unital quotient;
see \cref{dfn:Soft}.
We relate this property to the Hjelmborg-R{\o}rdam condition for stability and to property~(S) of Ortega-Perera-R{\o}rdam; 
see \cref{prp:RelSoft}.
We then give a spectral characterization of softness:

\begin{prpIntro}[\ref{prp:CharSoftElement}]
A positive element $a$ in a \ca{} $A$ is soft if and only if for every closed ideal $I \subseteq A$, either $a \in I$ or the spectrum of $a+I \in A/I$ has $0$ as a limit point.
\end{prpIntro}

\cref{sec:Shades} is devoted to the introduction of three tailored definitions of softness for Cuntz classes, which we call \emph{functional}, \emph{weak}, and \emph{strong} softness; 
see \cref{dfn:soft}.
Strong softness implies weak softness, which in turn implies functional softness, and all three notions agree in Cuntz semigroups of residually stably finite \ca{s};
see \cref{prp:soft}.
In that case, softness of an element in the \ca{} can be characterized in terms of its Cuntz class:

\begin{prpIntro}[\ref{prp:SoftCaVsCu}]
Let $A$ be a residually stably finite \ca{}.
Then an element $a\in A_+$ is soft if and only if its Cuntz class $[a] \in \Cu(A)$ is strongly soft (equivalently: weakly soft, functionally soft).
\end{prpIntro}

In order to connect the Global Glimm Problem to soft operators, we introduce the notion of \emph{abundance of soft elements} for a \ca{} $A$, which means that for every $a \in A_+$ and $\varepsilon >0$ there exists a positive, soft operator $b$ in the hereditary sub-\ca{} $\overline{aAa}$ such that the cut-down $(a-\varepsilon )_+$ is in the closed ideal generated by~$b$; 
see \cref{dfn:AbSoft}.
This definition is in the same spirit as the Global Glimm Property, which one could phrase as an `abundance of nilpotent elements';
see \cite[Theorem~3.6]{ThiVil22arX:Glimm}.
(We recall the definition of the Global Glimm Property at the beginning of \cref{sec:Glimm}.)

Although not every nilpotent element is soft, it is nevertheless true that the Global Glimm Property implies an abundance of soft elements;
see \cref{prp:2DivGivesStronglySoft}.
We also show that having an abundance of soft elements implies nowhere scatteredness (\cref{prp:AbSoftImplNWS}), which therefore breaks the Global Glimm Problem into two subquestions;
see \cref{qst:ReverseImplications}.

We introduce in \cref{dfn:AbSoft} the notion of having an \emph{abundance of strongly soft elements} for Cuntz semigroups.
Using that soft operators have strongly soft Cuntz classes, it follows that \ca{s} with an abundance of soft elements have scaled Cuntz semigroups with an abundance of soft elements.
We show that the converse also holds (\cref{prp:EquivAbundance}), although it  remains unclear if every strongly soft Cuntz class can be realized by a soft element (\cref{qst:RealizeSoftClass}).

Our notion of abundance of soft elements has close connections with the $2$-splitting property, which was introduced by Kirchberg and R{\o}rdam in \cite{KirRor15CentralSeqCharacters} to study when central sequence algebras have characters.
In \cref{prp:EquivAbundance}, we show that a \ca{} has an abundance of soft elements if and only if it has the \emph{hereditary $2$-splitting property} (see \cref{dfn:Her2Splitting}), a natural strengthening of the $2$-splitting property.

Using these characterizations, we provide a new description of the Global Glimm Property:

\begin{thmIntro}[\ref{prp:CharGlimm}]
A \ca{} has the Global Glimm Property if and only if it has and abundance of soft elements and its Cuntz semigroup is ideal-filtered.
\end{thmIntro}

In upcoming work with Asadi-Vasfi, \cite{AsaThiVil23pre:DimRcSoft}, we will exploit the abundance of soft elements to find bounds for some numerical \ca{ic} invariants, such as the radius of comparison and the Cuntz covering dimension introduced in \cite{ThiVil22DimCu}.

\section{Soft C*-algebras}

In this section, we introduce the notion of \emph{softness} for \ca{s};
see \cref{dfn:Soft}.
We provide useful characterizations of softness in \cref{prp:CharSoftAlg}, and we relate it to the Hjelmborg-R{\o}rdam condition for stability and to property (S) of Ortega-Perera-R{\o}rdam; 
see \cref{prp:RelSoft}.

We also prove basic permanence properties:
Softness passes to ideals, quotients and extensions (\cref{prp:PermIdealQuotExt}), to tensor products (\cref{prp:TensorProd}) and to inductive limits (\cref{prp:PermLimit}).

\begin{dfn}
\label{dfn:Soft}
We say that a \ca{} $A$ is \emph{soft} if for every proper, closed ideal $I\subseteq A$ the quotient $A/I$ is nonunital.
\end{dfn}

\begin{rmk}
\label{rmk:Soft}
Our notion of softness for \ca{s} is not related to that introduced by Farsi in \cite{Far02SoftCA}.
The latter is defined in terms of soft polynomial relations.
Our terminology is inspired by the close connection to soft elements in Cuntz semigroups (see \cref{prp:SoftCaVsCu}) and is therefore ultimately based on the notion of soft intervals introduced by Goodearl and Handelman in \cite{GooHan82Stenosis};
see also \cite[Section~5.3]{AntPerThi18TensorProdCu}.
\end{rmk}

The next result follows immediately from the definition.
It is, in some sense, the \ca{ic} analog of \cite[Proposition~5.3.16]{AntPerThi18TensorProdCu}.

\begin{prp}
\label{prp:SimpleCA}
A simple \ca{} is either unital or soft.
\end{prp}

\begin{ntn}
\label{ntn:RelationsCAlg}
Let $A$ be a \ca.
The annihilator of $a \in A_+$ is
\[
\Ann(a) := \big\{ x\in A : xa=ax=0 \big\}.
\]

We write $\Ann_A(a)$ whenever we want to stress that the annihilator is considered inside the algebra $A$. 
Note that $\Ann(a)$ is a hereditary sub-\ca{} of~$A$.

If $a,b \in A_+$, then:
\begin{itemize}
\item
we write $a \ll b$ if $b$ acts as a unit on $a$, that is, if $a=ab$;
\item
we write $a \lhd b$ to mean that $a$ belongs to the closed ideal generated by $b$, that is, if $a\in\overline{\linSpan} AbA$.
\item
we write $a \perp b$ if $ab = 0$.
\end{itemize} 
\end{ntn}

\begin{lma}
\label{prp:AnnSurjective}
Let $\pi \colon A \to B$ be a surjective \stHom{} between \ca{s}, and let $a,b \in A_+$ satisfy $a \ll b$.
Then $\Ann_B(\pi(b)) \subseteq \pi(\Ann_A(a))$.
\end{lma}
\begin{proof}
By adjoining a unit to $A$ if necessary, we may assume that $A$ is unital.
Let $x \in \Ann_B(\pi(b))$.
Lift $x$ to obtain $y \in A$ with $\pi(y) = x$.
Set $z := (1-b)y(1-b) \in A$.
Then $\pi(z) = x$.
Further, using that $a \ll b$, we have $a(1-b) = (1-b)a = 0$, and therefore $z \in \Ann_A(a)$.
\end{proof}

\begin{lma}
\label{prp:AnnProjection}
Let $A$ be a \ca, let $B \subseteq A$ be a soft sub-\ca, and let $p \in B$ be a projection.
Then $\Ann(p)$ is full in $A$.
\end{lma}
\begin{proof}
Let $I$ denote the closed ideal of $A$ generated by the annihilator $\Ann_A(p)$, and let $\pi \colon A \to A/I$ denote the quotient map.
Since $p \ll p$, we can apply \cref{prp:AnnSurjective} at the first step to obtain
\[
\Ann_{A/I}(\pi(p)) 
\subseteq \pi(\Ann_A(p))
\subseteq \pi(I) 
= \{0\}.
\]

Hence, the projection $\pi(p)$ has zero annihilator in $A/I$.
Assuming $\pi(p) \neq 0$, it follows that $\pi(p)$ is the unit of $A/I$ and, thus, of $\pi(B)$.
This contradicts the softness of $B$ and shows that $\pi(p) = 0$, which implies $A/I = \{0\}$ and so $I = A$.
\end{proof}

\begin{ntn}
Following the notation from \cite{BicKos19LLApproxUnits}, we set
\[
A_c := \big\{ a\in A_+ : a \ll b \text{ for some } b \in A_+ \big\}.
\]

The (nonclosed) ideal of $A$ generated by $A_c$ is called the Pedersen ideal -- it is the minimal dense ideal of $A$;
see \cite[Theorem~II.5.2.4]{Bla06OpAlgs}.

Note that $A_c$ is denoted by $F(A)$ in \cite{OrtPerRor12CoronaStability} and \cite{HirRorWin07CXAlgStabSSA}. We do not use this notation to avoid confusion with the central sequence algebra of $A$. In \cite{Bla06OpAlgs}, $A_c$ is denoted by $A_+^c$
\end{ntn}

\begin{lma}
\label{prp:AnnSoftSubalg}
Let $A$ be a \ca, let $B \subseteq A$ be a soft sub-\ca, and let $a \in B_c$.
Then $\Ann(a)$ is full in $A$.
\end{lma}
\begin{proof}
Let $I$ denote the closed ideal of $A$ generated by $\Ann_A(a)$.
We will show that $I=A$.
Choose $b \in B_+$ such that $a \ll b$.
Let $f,g \colon \RR \to [0,1]$ be given by
\[
f(t) = \begin{cases}
0, & t\in(-\infty,\tfrac{1}{2}] \\
2t-1, & t\in[\tfrac{1}{2},1] \\
1, & t\in[1,\infty),
\end{cases}
\andSep
g(t) = \begin{cases}
0, & t\in(-\infty,0] \\
2t, & t\in[0,\tfrac{1}{2}] \\
1, & t\in[\tfrac{1}{2},\infty).
\end{cases}
\]

Then $a \ll f(b) \ll g(b)$.

Let $J$ denote the closed ideal generated by $g(b)-f(b)$, and let $\pi_J\colon A\to A/J$ denote the quotient map.
Note that $g(b)-f(b) \in \Ann_A(a)$.
Thus, $J\subseteq I$.

Then $\pi_J(f(b))\ll\pi_J(g(b))=\pi_J(f(b))$, which shows that $\pi_J(f(b))$ is a projection.
Using that $B$ is soft, it easily follows that $\pi_J(B)$ is soft as well.
Then, by \cref{prp:AnnProjection}, the annihilator of $\pi_J(f(b))$ in $A/J$ is full.
Since $a \ll f(b)$, we can apply \cref{prp:AnnSurjective} to obtain
\[
\Ann_{A/J}(\pi_J(f(b)))
\subseteq \pi_J(\Ann_A(a)).
\]

It follows that $\pi_J(I)=A/J$.
Combined with $J \subseteq I$, we get $I=A$.
\end{proof}

The next result is similar to the Hjelmborg-R{\o}rdam characterization of stability from \cite{HjeRor98Stability}.

\begin{prp}
\label{prp:CharSoftAlg}
Let $A$ be a \ca.
Then the following are equivalent:
\begin{enumerate}
\item
$A$ is soft;
\item
for every $a\in A_c$, the algebra $\Ann(a)$ is full in $A$;
\item
for every $a\in A_c$, there exist $b \in A_+$ with $a \perp b$ and $a \lhd b$;
\item
for every $a\in A_+$ and $\varepsilon>0$ there exist $b,c\in A_+$ with $\|a-b\|<\varepsilon$, $b\perp c$, and $b \lhd c$.
\end{enumerate}
\end{prp}
\begin{proof}
It follows from \cref{prp:AnnSoftSubalg} that~(1) implies~(2).
To show that~(2) implies~(3), assume that $\Ann(a)$ is full.
This implies that $a$ belongs to the closed ideal generated by $\Ann(a)$, which allows us to choose a sequence $(b_n)_n$ in $\Ann(a)_+$ such that~$a$ belongs to the closed ideal generated by $\{b_n : n \in \NN\}$.
Set $b:=\sum_n \tfrac{1}{2^n\|b_n\|+1}b_n$.
Then $b$ belongs to $\Ann(a)$.
Further, for each $n$, we have $b_n\leq (2^n\|b_n\| +1)b$, which implies that $b_n$ belongs to the closed ideal generated by $b$. Consequently, $a$ belongs to the closed ideal generated by $b$.

To show that~(3) implies~(4), assume that~(3) holds, and let $a\in A_+$ and $\varepsilon>0$.
Set $b:=(a-\tfrac{\varepsilon}{2})_+$.
Then $b\in A_c$ and $\|a-b\|=\tfrac{\varepsilon}{2}<\varepsilon$.
Applying~(3), we obtain $c\in A_+$ such that $b\perp c$ and $b \lhd c$.

Finally, to show that~(4) implies~(1), assume that~(4) holds and let $I\subseteq A$ be a proper closed ideal.
To reach a contradiction, assume that $A/I$ is unital.
Choose $a\in A_+$ such that its image in $A/I$ is the unit.
Applying~(4) with $\varepsilon=1$, we obtain $b,c \in A_+$ such that $\|a-b\|<1$ and $b \lhd c$.
It follows that the image of $b$ in the quotient $A/I$ is at distance less than $1$ to the unit, and is therefore strictly positive.
Thus, the image of $c$ in $A/I$ is zero, that is, $c\in I$.
We get $b\in I$, which implies that $A/I=\{0\}$.
\end{proof}


\begin{pgr}
\label{pgr:CuntzSgp}
Given two positive elements $a,b$ in a \ca{} $A$, we say that $a$ is \emph{Cuntz subequivalent} to $b$, written $a\precsim b$, if there exists a sequence $(r_n)_n$ in $A$ such that $a=\lim_n r_n br_n^*$. 
Further, one says that $a$ is \emph{Cuntz equivalent} to $b$, and writes $a\sim b$, if $a\precsim b\precsim a$.
 
The \emph{Cuntz semigroup} $\Cu (A)$ of a \ca{} $A$ is the quotient $(A\otimes\KK)_+/\!\sim$ with the order induced by the Cuntz subequivalence and the addition induced by addition of orthogonal elements;
we refer to \cite{AraPerTom11Cu, AntPerThi18TensorProdCu} for details.

We let $\QT(A)$ denote the cone of $[0,\infty]$-valued, lower-semicontinuous $2$-quasi\-traces on $(A\otimes\KK)_+$.
Each element $[a] \in \Cu (A)$ induces a linear, lower semicontinuous map $\wh{[a]} \colon \QT(A) \to [0,\infty]$ by $\wh{[a]}(\tau) := \lim_{n \to \infty} \tau(a^{1/n})$ for $\tau \in \QT(A)$;
see \cite{EllRobSan11Cone, Rob13Cone, AntPerRobThi21Edwards}.
\end{pgr}

In \cite{HjeRor98Stability}, Hjelmborg and R{\o}rdam introduced a property for \ca{s} that characterizes stability in the $\sigma$-unital case.
We will say that a \ca{} has the \emph{Hjelmborg-R{\o}rdam property} if, for every $a\in A_c$, there exists $b\in A_+$ such that $a\perp b$, and such that $a=x^*x$ and $b=xx^*$ for some $x\in A$.

\begin{prp}
\label{prp:CharHR}
A \ca{} $A$ has the Hjelmborg-R{\o}rdam property if and only if for every $a\in A_c$ there exists $b\in A_+$ such that $a\perp b$ and $[a]\leq[b]$ in $\Cu(A)$.
\end{prp}
\begin{proof}
For every $x\in A$, we have $[x^*x]=[xx^*]$ in $\Cu (A)$. 
This shows the forward implication.

Conversely, let $a\in A_c$ and take $\varepsilon >0$. 
By assumption, there exists $b\in A_+$ such that $a\perp b$ and $[a]\leq [b]$ in $\Cu(A)$. 
Thus, it follows from \cite[Proposition~2.4]{Ror92StructureUHF2} that there exists $x\in A$ satisfying 
\[
(a-\varepsilon )_+=xx^*, \andSep x^*x\in \overline{bAb}.
\]

In particular, note that $xx^*$ and $x^*x$ are orthogonal.
Applying \cite[Proposition~2.2]{HjeRor98Stability}, we see that $A$ has the Hjelmborg-R{\o}rdam property.
\end{proof}

For elements $[a],[b]$ in a Cuntz semigroup, one writes $[a]<_s [b]$ if $(n+1)[a]\leq n[b]$ for some $n\in\NN$.
Following \cite[Definition~4.1]{OrtPerRor12CoronaStability}, we say that a \ca{} $A$ has \emph{property~(S)} if for every $a\in A_c$ there exists $b\in A_+$ such that $a\perp b$ and $[a]<_s[b]$ in $\Cu(A)$.

\begin{prp}
\label{prp:CharS}
A \ca{} $A$ has property~(S) if and only if for every $a \in A_c$ there exists $b \in A_+$ such that $a \perp b$ and $\wh{[a]} \leq \wh{[b]}$.
\end{prp}
\begin{proof}
Let $[a],[b]$ in $\Cu (A)$. 
In general, one has that $[a]<_s [b]$ implies $\wh{[a]}\leq\wh{[b]}$.  
This shows the forward implication.

For the converse, let $a \in A_c$.
Using an argument as in the proof of \cref{prp:AnnSoftSubalg}, we can find $b \in A_c$ such that $a \ll b$.
Applying the assumption for $b$, we obtain $c \in A_+$ such that $b \perp c$ and $\wh{[b]} \leq \wh{[c]}$.

Note that $[a]$ is way-below $[b]$ in $\Cu(A)$, denoted $[a] \ll [b]$; see \cref{pgr:CuSgps} for details.
This allows us to choose $\varepsilon>0$ such that $[a] \ll [(b-\varepsilon)_+]$.
Using that $[(b-\varepsilon)_+] \ll [b]$, we can apply \cite[Lemma~2.2.5]{Rob13Cone} to deduce that $\wh{[(b-\varepsilon)_+]}$ is way-below $2\wh{[b]}$ (and hence also $2\wh{[c]}$) in the semigroup of linear, lower-semicontinuous functions $\QT(A) \to [0,\infty]$.
This allows us to choose $\delta>0$ such that $\wh{[(b-\varepsilon)_+]} \leq 2\wh{[(c-\delta)_+]}$.

Note that $b+(c-\delta)_+$ belongs to $A_c$.
Applying the assumption again, we obtain $d \in A_+$ such that
\[
b+(c-\delta)_+ \perp d, \andSep
\wh{[b]} + \wh{[(c-\delta)_+]} \leq \wh{[d]},
\]
where at the second inequality we have used that $b\perp (c-\delta)_+$.

Then
\[
2\wh{[(b-\varepsilon)_+]} 
\leq \wh{[b]} + \wh{[(b-\varepsilon)_+]} 
\leq \wh{[b]} + 2 \wh{[(c-\delta)_+]}
\leq \wh{[(c-\delta)_+]} + \wh{[d]}.
\]

Thus, $\wh{[(b-\varepsilon)_+]} <_s  \wh{[(c-\delta)_+]} + \wh{[d]}$.
Using at the first step that $[a] \ll [(b-\varepsilon)_+]$ and applying \cite[Theorem~5.2.13]{AntPerThi18TensorProdCu}, we get
\[
[a] 
<_s [(c-\delta)_+] + [d]
= [(c-\delta)_+ + d].
\]

Since $a \perp (c-\delta)_+ + d$, we see that $(c-\delta)_+ + d$ has the desired properties.
\end{proof}

\begin{rmk}
\label{rmk:RelSoft}
For positive elements $a$ and $b$ in a \ca{}, we have $a \lhd b$ if and only if $[a] \leq \infty [b]$ in $\Cu(A)$;
see, for example, \cite[Section~5.1]{AntPerThi18TensorProdCu}.

Therefore, Propositions~\ref{prp:CharHR}, \ref{prp:CharS} and~\ref{prp:CharSoftAlg} show that the Hjelmborg-R{\o}rdam property, property~(S), and softness for a \ca{} $A$, can be characterized in very similar ways:
For every $a\in A_c$ there exists $b\in A_+$ such that $a\perp b$ and
\begin{itemize}
\item
$[a]\leq[b]$ for the Hjelmborg-R{\o}rdam property;
\item
$\wh{[a]}\leq\wh{[b]}$ for property~(S);
\item
$[a]\leq\infty[b]$ for softness.
\end{itemize}

For elements $[a],[b]$ in a Cuntz semigroup, we have the following implications:
\[
[a] \leq [b] \quad\Rightarrow\quad
\wh{[a]}\leq\wh{[b]} \quad\Rightarrow\quad
[a]\leq\infty[b].
\]

This proves the next result.
\end{rmk}

\begin{prp}
\label{prp:RelSoft}
For every \ca{}, the following implications hold:
\[
\text{Hjelmborg-R{\o}rdam property} \quad\Rightarrow\quad
\text{property~(S)} \quad\Rightarrow\quad
\text{softness}.
\]
\end{prp}

\begin{prp}[{\cite[Proposition~4.5]{OrtPerRor12CoronaStability}}]
A separable \ca{} has property~(S) if and only if it is soft and has no bounded quasitraces.
\end{prp}

%

We now turn to permanence properties of softness.
With view towards the close connection between softness, the Global Glimm Property, and nowhere scatteredness discussed in \cref{sec:Glimm}, we remark that similar permanence properties hold for nowhere scatteredness (\cite[Section~4]{ThiVil21arX:NowhereScattered}) and the Global Glimm Property (\cite[Section~3]{ThiVil22arX:Glimm}).

\begin{prp}
\label{prp:PermIdealQuotExt}
Let $A$ be \ca{}, and let $I\subseteq A$ be a closed ideal.
Then $A$ is soft if and only if $I$ and $A/I$ are soft.
\end{prp}
\begin{proof}
Assume first that $A$ is soft, and let $I$ be a closed ideal of $A$. 
Clearly, $A/I$ is soft. 
To see that $I$ is also soft, assume for the sake of contradiction that there is a proper ideal $J \subseteq I$ such that $I/J$ is unital, with unit $e \in I/J$.
Then $I/J$ is a unital ideal of $A/J$, which implies that $A/J$ decomposes as the direct sum of $I/J$ and the ideal $K := \Ann_{A/J}(e)$.
It follows that $I/J = (A/J)/K$ is a proper, unital quotient of $A$, a contradiction.

Now assume that $I$ and $A/I$ are soft. 
Arguing once again by contradiction, assume that there exists a proper ideal $J \subseteq A$ such that $A/J$ is unital. 
Then, $A/(I+J)$ is isomorphic to a quotient of both $A/I$ and $A/J$. 
Since $A/I$ is soft and $A/J$ is unital, we get that $A/(I+J)$ is both soft and unital and thus $A=I+J$.

It follows that $A/J \cong I/(I\cap J)$. 
Using that $I$ is soft, we see that $A=J$. 
This contradicts the fact that $J$ is proper.
\end{proof}

One says that a \ca{} $A$ is \emph{approximated} by a family $(A_\lambda)_{\lambda\in\Lambda}$ of sub-\ca{s} $A_\lambda \subseteq A$ if for every finite collection $a_1, \ldots, a_n \in A$ and $\varepsilon > 0$ there exists $\lambda \in \Lambda$ and $b_1, \ldots, b_n \in A_\lambda$ such that $\|a_j-b_j\| < \varepsilon$ for $j=1,\ldots,n$.

\begin{prp}
\label{prp:PermApprox}
A \ca{} is soft whenever it is approximated by a family of soft sub-\ca{s}.
\end{prp}
\begin{proof}
Let $A$ be a \ca{}, and let $A_\lambda\subseteq A$ be soft sub-\ca{s} that approximate $A$.
To show that $A$ is soft, we verify condition~(4) in \cref{prp:CharSoftAlg}.
Let $a\in A_+$ and $\varepsilon>0$.
Choose $\lambda$ and $\bar{a}\in A_\lambda$ such that $\|a-\bar{a}\|<\tfrac{\varepsilon}{2}$.
Taking additional care, we may assume $\bar{a}$ to be positive; see, for example, the arguments in the proof of \cite[Proposition~3.7]{ThiVil21DimCu2}.

Using that $A_\lambda$ is soft, we can apply \cref{prp:CharSoftAlg} to obtain $b,c\in (A_\lambda)_+$ such that $\|\bar{a}-b\|<\tfrac{\varepsilon}{2}$, $b \perp c$, and $b \in \overline{\linSpan}A_\lambda c A_\lambda$.
Then
\[
\| a-b \|
\leq \| a-\bar{a} \| + \| \bar{a} - b \|
<\varepsilon, \andSep
b\in\overline{\linSpan}A_\lambda c A_\lambda \subseteq \overline{\linSpan}A c A
\]
which shows that $b$ and $c$, viewed in $A$, have the desired properties.
\end{proof}

\begin{prp}
\label{prp:PermLimit}
Softness passes to inductive limits.
\end{prp}
\begin{proof}
This follows from \cref{prp:PermIdealQuotExt,prp:PermApprox} as in the proof of \cite[Proposition~2.4]{Thi21GenRnk}.
We include some details for the convenience of the reader.
Let $A$ be the inductive limit of an inductive system of soft \ca{s} $A_j$.
For each $j$, consider the natural map $A_j\to A$ and its image $B_j\subseteq A$.
Since $B_j$ is a quotient of $A_j$, it is soft by \cref{prp:PermIdealQuotExt}.
Now, using that $A$ is approximated by the family $B_j$, it follows from \cref{prp:PermApprox} that $A$ is soft.
\end{proof}

\begin{prp}
\label{prp:TensorProd}
Let $A,B$ be \ca{s}.
Then the following are equivalent:
\begin{enumerate}
\item
$A$ or $B$ is soft;
\item
The maximal tensor product $A\otimes_{\mathrm{max}}B$ is soft;
\item
The tensor product $A\otimes_{\varrho}B$ is soft for some (equivalently: any) cross-norm~$\varrho$;
\item
The minimal tensor product $A\otimes_{\mathrm{min}}B$ is soft.
\end{enumerate}
\end{prp}
\begin{proof}
For every cross-norm $\varrho$, the tensor product $A\otimes_{\varrho}B$ is a quotient of $A\otimes_{\mathrm{max}}B$, and $A\otimes_{\mathrm{min}}B$ is a quotient of $A\otimes_{\varrho}B$.
Since softness passes to quotients by \cref{prp:PermIdealQuotExt}, we see that~(2) implies~(3), and that~(3) implies~(4).

Let us show that~(4) implies~(1).
To reach a contradiction, assume that neither $A$ nor $B$ is soft.
Then there are proper closed ideals $I \subseteq A$ and $J \subseteq B$ such that $A/I$ and $B/J$ are unital.
The quotient maps $A \to A/I$ and $B \to B/J$ induce a natural surjective $\ast$-homomorphism $A\otimes_{\mathrm{min}}B \to (A/I)\otimes_{\mathrm{min}}(B/J)$ (see, for example, \cite[II.9.6.6]{Bla06OpAlgs}).
Since $(A/I)\otimes_{\mathrm{min}}(B/J)$ is unital, this shows that $A\otimes_{\mathrm{min}}B$ is not soft.

Finally, let us show that~(1) implies~(2).
To reach a contradiction, assume that $A\otimes_{\mathrm{max}}B$ is not soft.
Then there exists a $\ast$-representation $\pi \colon A\otimes_{\mathrm{max}}B \to \Bdd(H)$ on some Hilbert space $H$ such that the identity $1 \in \Bdd(H)$ is contained in the image of~$\pi$.
By \cite[Theorem~II.9.2.1]{Bla06OpAlgs} there exist unique $\ast$-representations $\pi_A \colon A \to \Bdd(H)$ and $\pi_B \colon B \to \Bdd(H)$ such that $\pi(a\otimes b) = \pi_A(a)\pi_B(b) = \pi_B(b)\pi_A(a)$ for all $a \in A$ and $b \in B$.

Choose approximate units $(a_\lambda)_\lambda$ for $A$ and $(b_\mu)_\mu$ for $B$.
Then $(a_\lambda\otimes b_\mu)_{\lambda,\mu}$ is an approximate unit for $A\otimes_{\mathrm{max}}B$ and it follows that $\pi(a_\lambda\otimes b_\mu)$ converges to $1 \in \Bdd(H)$ in norm.
Thus, by setting $a := a_\lambda$ and $b := b_\mu$ for sufficiently large $\lambda$ and~$\mu$, the operator $\pi(a \otimes b)$ is invertible.
Using that $\pi(a\otimes b) = \pi_A(a)\pi_B(b) = \pi_B(b)\pi_A(a)$, we deduce that $\pi_A(a)$ and $\pi_B(b)$ are invertible.
This implies that $1$ belongs to the image of $\pi_A$ and $\pi_B$.
Hence, $\ker(\pi_A)$ and $\ker(\pi_B)$ are proper closed ideals of $A$ and~$B$ with unital quotients, showing that neither $A$ nor $B$ is soft.
\end{proof}

The next result shows that softness is a stable property.

\begin{prp}
Let $A$ be a \ca, and let $n\geq 1$.
Then $A$ is soft if and only if $M_n(A)$ is.
\end{prp}
\begin{proof}
The forward implication follows directly from \cref{prp:TensorProd}.

Conversely, assume that $M_n(A)$ is soft. Thus, using \cref{prp:TensorProd}, we know that either $M_n (\mathbb{C})$ or $A$ is soft. Since $M_n (\mathbb{C})$ is not soft, we must have $A$ soft, as desired.
\end{proof}

\begin{qst}
Does softness satisfy the L\"{o}wenheim-Skolem condition?
Here, a property $\mathcal{P}$ of \ca{s} is said to satisfy the \emph{L\"{o}wenheim-Skolem condition} if for every \ca{} $A$ satisfying $\mathcal{P}$, there exists a $\sigma$-complete, cofinal subset of separable sub-\ca{s} of $A$ that each satisfy $\mathcal{P}$;
see \cite[Section~4]{ThiVil21arX:NowhereScattered} for further details.
\end{qst}

\section{Soft operators}

In this section, we introduce and study the notion of \emph{softness} for operators in \ca{s};
see \cref{dfn:SoftElement}.
We will see in \cref{sec:AbSoftEl} that this concept is closely related to the notion of softness for elements in a Cuntz semigroup.

\begin{dfn}
\label{dfn:SoftElement}
We say that an element $x$ in \ca{} $A$ is \emph{soft} if the hereditary sub-\ca{} $\overline{x^*Ax}$ is soft.
\end{dfn}

The next result shows that we could equivalently use $\overline{xAx^*}$ instead of $\overline{x^*Ax}$ in the above definition.
We give further characterizations of softness for positive operators in \cref{prp:CharSoftElement} below.

\begin{prp}
\label{prp:SymmetryDfnSoftElement}
Let $A$ be a \ca, and $x \in A$.
Then the following are equivalent:
\begin{enumerate}
\item
The element $x$ is soft.
\item
The element $x^*$ is soft.
\item
The element $x^*x$ is soft.
\item
The element $xx^*$ is soft.
\end{enumerate}
\end{prp}
\begin{proof}
By \cite[Proposition~II.3.4.2]{Bla06OpAlgs}, we have $\overline{x^*Ax} = \overline{x^*xAx^*x}$.
This shows that~(1) and~(3) are equivalent, and similarly that~(2) and~(4) are equivalent.
The equivalence of~(1) and~(2) follows using that $\overline{x^*Ax}$ and $\overline{xAx^*}$ are isomorphic.
Indeed, if $x=v|x|$ is the polar decomposition of $x$ in $A^{**}$, then $a\mapsto vav^*$ is a well-defined isomorphism $\overline{x^*Ax} \to \overline{xAx^*}$;
see also \cite[Lemma~4.2]{OrtRorThi11CuOpenProj}.
\end{proof}

\begin{rmk}
\label{rmk:SoftElement}
The property of softness for operators depends on the containing \ca{}.
For example, if $a$ is a strictly positive element in a simple, non-unital \ca{} $A$, then $a$ is soft in $A$ since $\overline{aAa}=A$, which is soft by \cref{prp:SimpleCA}, while $a$ is not soft as an element of the commutative \ca{} $C^*(a)$.
\end{rmk}

Following Robert, \cite{Rob16RmksZstblProjless}, we say that a \ca{} is \emph{projectionless} if none of its quotients contains a nonzero projection.

\begin{prp}
A \ca{} $A$ is projectionless if and only if every element of $A$ is soft.
\end{prp}
\begin{proof}
Let $A$ be a \ca.
Assume first that $A$ is projectionless.
Given any hereditary sub-\ca{} $B \subseteq A$, and a closed ideal $I \subseteq B$, let $J$ be the closed ideal of $A$ generated by $I$.
Then the quotient $B/I$ is naturally isomorphic to a hereditary sub-\ca{} of $A/J$.
Since $A/J$ contains no nonzero projections, $B/I$ cannot be unital and nonzero.
This shows that every hereditary sub-\ca{} of $A$ is soft, and hence so is every element of $A$.

Conversely, assume that $A$ is not projectionless.
Choose a closed ideal $I \subseteq A$ and a nonzero projection $p \in A/I$.
Pick a positive lift $a \in A$ of $p$.
Then $a$ is not soft since $\overline{aAa}$ has the nonzero, unital quotient $p(A/I)p$.
\end{proof}

\begin{ntn}
\label{ntn:g-epsilon}
Given $\varepsilon>0$, we let $g_\varepsilon\colon\RR\to\RR$ be the continuous function given by 
\[
g_\varepsilon(t) = \begin{cases}
0, & \text{if } t\in(-\infty,0]\cup[\varepsilon,\infty) \\
t, & \text{if } t\in[0,\tfrac{\varepsilon}{2}] \\
\varepsilon-t, & \text{if } t\in[\tfrac{\varepsilon}{2},\varepsilon] \\
\end{cases}
\]
\end{ntn}

The equivalence of~(1) and~(3) in the next result is analogous to Bosa's characterization of stable elements; 
see \cite[Lemma~2.3]{Bos22StableElements}.

\begin{prp}
\label{prp:CharSoftElement}
Let $A$ be a \ca, and $a\in A_+$.
Then the following are equivalent:
\begin{enumerate}
\item
The element $a$ is soft.
\item
For every closed ideal $I \subseteq A$, either $a \in I$ or the spectrum of $a+I \in A/I$ has $0$ as a limit point.
\item
We have $a \lhd g_\varepsilon(a)$ for every $\varepsilon>0$.
\item
For every $\varepsilon>0$ there exists $b\in\overline{aAa}_+$ with $b \perp (a-\varepsilon)_+$ and $a \lhd b$.
\end{enumerate}
\end{prp}
\begin{proof}
Let us show that~(1) implies~(2).
Assuming that~$a$ is soft, let $I \subseteq A$ be a closed ideal such that $a \notin I$.
To reach a contradiction, assume that $0$ is isolated in the spectrum of $a+I$.
Then the characteristic function of $(0,\infty)$ is continuous on the spectrum of $a+I$, and applying functional calculus we obtain a projection in~$A/I$ that generates the same hereditary sub-\ca{} as $a+I$.
Thus, $\overline{aAa}$ has a nonzero, unital quotient, which is the desired contradiction.

Let us show that~(2) implies~(3).
Assuming~(2), let $\varepsilon>0$.
Set $I:=\overline{\linSpan} Ag_\varepsilon(a)A$. 
Then the spectrum of~$a+I$ is contained in $\{0\}\cup[\varepsilon,\infty)$.
By assumption, this implies that $a \in I$, and thus $a \lhd g_\varepsilon(a)$.

It is clear that~(3) implies~(4).
To show that~(4) implies~(1), let $I\subseteq \overline{aAa}$ be a proper closed ideal, and let $\pi\colon \overline{aAa}\to \overline{aAa}/I$ be the quotient map.
To reach a contradiction, assume that $\overline{aAa}/I$ is unital and $\overline{aAa}/I \neq \{0\}$.
Since $a$ is strictly positive, $\pi(a)$ is invertible in $\overline{aAa}/I$, and we obtain $\varepsilon>0$ such that the spectrum of $\pi(a)$ is contained in $[\varepsilon,\infty)$. 

Applying the assumption for $\tfrac{\varepsilon}{2}$, we obtain a full element $b\in {\overline{aAa}}_+$ such that $b\perp(a-\tfrac{\varepsilon}{2})_+$.
We have $\pi((a-\tfrac{\varepsilon}{2})_+)=(\pi(a)-\tfrac{\varepsilon}{2})_+$, whose spectrum is contained in $[\tfrac{\varepsilon}{2},\infty)$.
Thus, $\pi((a-\tfrac{\varepsilon}{2})_+)$ is strictly positive in $\overline{aAa}/I$, which implies that $\pi(b)=0$, a contradiction.
\end{proof}

A completely positive map between \ca{s} is said to have order-zero if it preserves zero-products among positive elements.
By \cite[Corollary~4.8]{GarThi22arX:WeightedHomo}, the class of completely positive, order zero maps can equivalently be described as the class of positive, zero-product preserving maps.

It remains unclear if one can remove the positivity assumption for the soft element in the next result;
see \cref{qst:PresSoft}.

\begin{prp}
\label{prp:OrderZeroMapPresSoft}
Let $\varphi \colon A \to B$ be a completely positive, order-zero map between \ca{s}, and let $a \in A_+$ be soft.
Then $\varphi(a)$ is soft.
\end{prp}
\begin{proof}
We use the characterization of softness from \cref{prp:CharSoftElement}(4).
So let $\varepsilon>0$.
We need to find a positive element $b \in \overline{\varphi(a)B\varphi(a)}$ orthogonal to $(\varphi(a)-\varepsilon)_+$ and such that $\varphi(a) \lhd b$.

Using that $a$ is soft, and applying \cref{prp:CharSoftElement}(3), we see that $a \lhd g_\varepsilon(a)$.
Set $b := \varphi(g_\varepsilon(a))$.
We have $g_\varepsilon(a) \leq a$, and therefore $b = \varphi(g_\varepsilon(a)) \leq \varphi(a)$, and thus $b \in \overline{\varphi(a)B\varphi(a)}$.

Further, we have $g_\varepsilon(a) \perp (a-\varepsilon)_+$, and since $\varphi$ preserves orthogonality among positive elements, we see that $b \perp \varphi((a-\varepsilon)_+)$.
As noted in the proof of \cite[Proposition~2.2.7]{AntPerThi18TensorProdCu}, we have $(\varphi(a)-\varepsilon)_+ \leq \varphi((a-\varepsilon)_+)$, which implies that $b$ is orthogonal to $(\varphi(a)-\varepsilon)_+$.

Let $x,y \in A_+$ with $x \lhd y$.
Then $[x] \leq \infty [y]$ in $\Cu(A)$, and since $\varphi$ naturally induces a generalized \CuMor{} $\Cu(A)\to\Cu(B)$ by \cite[Corollary~4.5]{WinZac09CpOrd0} (see also \cite[3.2.5]{AntPerThi18TensorProdCu}), we obtain that $[\varphi(x)] \leq \infty [\varphi(y)]$ in $\Cu(B)$, whence $\varphi(x) \lhd \varphi(y)$.
Applying this argument for $a \lhd g_\varepsilon(a)$, we deduce that $\varphi(a) \lhd b$, which shows that $b$ has the desired properties.
\end{proof}

Applying the above result for the quotient map by a closed ideal, we get:

\begin{prp}
\label{prp:SoftImageInQuotient}
Let $I \subseteq A$ be a closed ideal in a \ca{} $A$, and let $a \in A_+$ be soft.
Then $a+I \in A/I$ is soft.
\end{prp}

\begin{qst}
\label{qst:PresSoft}
Let $\varphi \colon A \to B$ be a completely positive, order-zero map between \ca{s}.
Given a (not necessarily positive) soft element $x \in A$, is $\varphi(x)$ soft?
\end{qst}

\section{Shades of softness in Cu-semigroups}
\label{sec:Shades}

In this section, we introduce the concepts of strong softness, weak softness and functional softness for elements in \CuSgp{s};
see \cref{dfn:soft}.
(We recall the definition of \CuSgp{s} in \cref{pgr:CuSgps} below.)
We show that every strongly soft element is weakly soft, and that every weakly soft element is functionally soft. 
For residually stably finite \CuSgp{s}, we prove that these three notions coincide;
see \cref{prp:soft}. 

Functionally soft elements were first considered in \cite[Definition~5.3.1]{AntPerThi18TensorProdCu}, where they are just called `soft' elements.
Strongly soft elements were implicitly also considered in \cite[Section~5.3]{AntPerThi18TensorProdCu}.
We use the term `functionally soft' to clarify the difference with the other notions.

We analyze the relation between soft elements in a \ca{} and (strongly) soft elements in its Cuntz semigroup in \cref{prp:SoftCaVsCu}.
In particular, in a stably finite \ca{}, a positive element is soft if and only if its Cuntz class is.

\begin{pgr}
\label{pgr:CuSgps}
Given two elements $x,y$ in a partially ordered monoid where suprema of increasing sequences exist, recall that we write $x\ll y$ if, for every increasing sequence $(z_n)_n$ such that $y\leq \sup_n z_n$, there exists $n\in\NN$ such that $x\leq z_n$.
 
Following \cite{CowEllIva08CuInv}, one says that a partially ordered, commutative monoid $S$ is a \CuSgp{} if it satisfies the conditions below:
\begin{enumerate}
\item[\axiomO{1}] 
Every increasing sequence in $S$ has a supremum.
\item[\axiomO{2}] 
Every element in $S$ can be written as the supremum of a $\ll$-increasing sequence.
\item[\axiomO{3}] 
Given $x',x,y',y\in S$ with $x'\ll x$ and $y'\ll y$, one has $x'+y'\ll x+y$.
\item[\axiomO{4}] 
Given two increasing sequences $(x_n)_n ,(y_n)_n$ in $S$, one has $\sup_n (x_n +y_n)=\sup_n x_n + \sup_n y_n$.
\end{enumerate}
 
As shown in \cite{CowEllIva08CuInv}, the Cuntz semigroup $\Cu (A)$ of any \ca{} $A$ is a \CuSgp{}. 
Further, it was subsequently shown in \cite{AntPerThi18TensorProdCu}, \cite{Rob13Cone} and \cite{AntPerRobThi21Edwards} respectively that $\Cu (A)$ also satisfies the following properties:
\begin{enumerate}
\item[\axiomO{5}]
Given $x',x,y',y,z\in S$ such that $x'\ll x$, $y'\ll y$ and $x+y\leq z$, there exists $c\in S$ with $x'+c\leq z\leq x+c$ and $y'\ll c$.  
\item[\axiomO{6}] 
Given $x',x,y,z\in S$ such that $x'\ll x\leq y+z$, there exist elements $e,f$ such that $e\leq x,y$, $f\leq x,z$ and $x'\leq e+f$.
\item[\axiomO{7}] 
Given $x',x,y',y,z\in S$ such that $x'\ll x\leq z$ and $y'\ll y\leq z$, there exists $w\in S$ satisfying $x',y'\ll w\leq z,x+y$.
\end{enumerate}

We will often use the equivalent formulation of \axiomO{6} that for $x'\ll x \leq y_1+\ldots+y_n$ delivers $e_1,\ldots,e_n$ such that $x'\leq e_1+\ldots+e_n$, and such that $e_j\leq x,y_j$ for all~$j$.
Similarly, we will also use an equivalent formulation of \axiomO{7} where for given $x_j' \ll x_j \leq z$ for $j=1,\ldots,n$ there exists $w \in S$ satisfying $x_j'\ll w \leq z,x_1+\ldots+x_n$.

Recently, an additional property that the Cuntz semigroup of every \ca{} satisfies, termed \axiomO{8}, has been uncovered;
see \cite[Theorem~7.4]{ThiVil21arX:NowhereScattered}.
\end{pgr}

\begin{dfn}
\label{dfn:soft}
Let $S$ be a \CuSgp, and let $x\in S$.
We say that $x$ is \emph{strongly soft} if for every $x'\in S$ satisfying $x'\ll x$ there exists $t\in S$ such that
\[
x'+t\ll x, \andSep x'\ll\infty t.
\]

We say that $x$ is \emph{weakly soft} if for every $x'\in S$ satisfying $x'\ll x$ there exists $n\geq 1$ and $t_1,\ldots,t_n\in S$ such that
\[
x'+t_j\ll x \quad \text{ for } j=1,\ldots,n, \andSep x'\ll t_1+\ldots+t_n.
\]

We say that $x$ is \emph{functionally soft} if for every $x'\in S$ satisfying $x'\ll x$ there exists $n\geq 1$ such that
\[
(n+1)x' \ll nx.
\]
\end{dfn}

\begin{rmk}
In \cite[Definition~5.3.1]{AntPerThi18TensorProdCu}, an element $x$ in a \CuSgp{} $S$ is defined to be `soft' if for every $x'\in S$ satisfying $x'\ll x$ there exists $n\in\NN$ such that $(n+1)x'\leq nx$.
Note that $x$ is `soft' in this sense if and only if $x$ is functionally soft in the sense of \cref{dfn:soft}. 
Indeed, the backward implication is clear.
Conversely, if $x$ is `soft' and $x'\ll x$, then choose $x''$ satisfying $x'\ll x''\ll x$.
Applying the definition we get $n$ such that $(n+1)x''\leq nx$, and thus $(n+1)x'\ll nx$.
\end{rmk}

The notions from \cref{dfn:soft} are closely related to the concept of \emph{pure noncompactness} introduced in \cite{EllRobSan11Cone}, and its generalization \emph{weak pure noncompactness} from \cite{AntPerThi18TensorProdCu}.
After recalling the definitions, we will see how these notions are related.
We refer to \cite[Section~5.1]{AntPerThi18TensorProdCu} for details on ideals and quotients of \CuSgp{s}.

\begin{dfn}
\label{dfn:pnc}
Let $S$ be a \CuSgp, and let $x\in S$.
One says that $x$ is \emph{purely noncompact} if for every ideal $I \subseteq S$ such that the image $x_I$ of $x$ in the quotient $S/I$ is compact, we have $2x_I = x_I$.

One says that $x$ is \emph{weakly purely noncompact} if for every ideal $I \subseteq S$ such that~$x_I$ is compact, there exists $n \in \NN$ such that $(n+1)x_I = nx_I$.
\end{dfn}

We will say that a \CuSgp{} $S$ is \emph{stably finite} if for all $x,y\in S$ with $x+y\ll x$ we have $y=0$. 
Note that this is more restrictive than the definition in \cite{AntPerThi18TensorProdCu}, although both notions agree if $S$ is simple.
Further, we will say that $S$ is \emph{residually stably finite} if every quotient of $S$ is stably finite.

A stably finite \ca{} has a stably finite Cuntz semigroup, and a residually stably finite \ca{} has a residually stably finite Cuntz semigroup.

\begin{prp}
\label{prp:soft}
Let $S$ be a \CuSgp, and let $x\in S$.
If $x$ is strongly soft, then $x$ is weakly soft, which in turn implies that $x$ is functionally soft and purely noncompact.
Further, $x$ is weakly purely noncompact whenever it is purely noncompact or functionally soft.

If $S$ satisfies \axiomO{5}, then $x$ is functionally soft if and only if $x$ is weakly purely noncompact.
If $S$ satisfies \axiomO{5} and is residually stably finite, then all notions from \cref{dfn:soft,dfn:pnc} are equivalent.

The implications are shown in the following diagram:
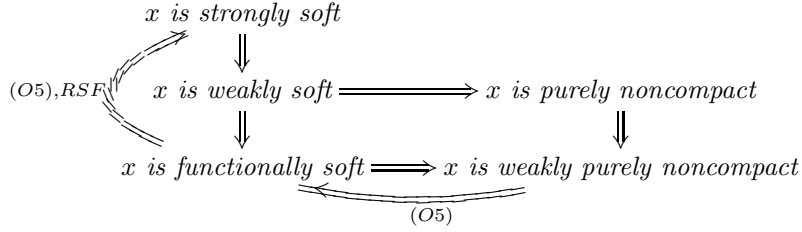
\[
\xymatrix@R-10pt{
\text{$x$ is strongly soft} \ar@{=>}[d] \\
\text{$x$ is weakly soft} \ar@{=>}[d] \ar@{=>}[r]
& \text{$x$ is purely noncompact} \ar@{=>}[d]  
\\
\text{$x$ is functionally soft} \ar@{=>}[r] \ar@/^4pc/@{==>}[uu]^{\axiomO{5},RSF}
& \text{$x$ is weakly purely noncompact} \ar@/^1pc/@{==>}[l]^{\axiomO{5}}
}
\]
\end{prp}
\begin{proof}
It is easy to see that strong softness implies weak softness, and that pure noncompactness implies weak pure noncompactness.

To show that weak softness implies functional softness, assume that $x$ is weakly soft, and let $x'\in S$ satisfy $x'\ll x$.
We obtain $n\geq 1$ and $t_1,\ldots,t_n\in S$ such that
\[
x'+t_j\ll x \quad \text{ for } j=1,\ldots,n, \andSep x'\ll t_1+\ldots+t_n.
\]

Then
\[
(n+1)x' \leq nx' + t_1+\ldots+t_n \ll nx.
\]

To show that weak softness implies pure noncompactness, assume that $x$ is weakly soft, and let $I \subseteq S$ be an ideal such that $\pi(x) \ll \pi(x)$, where $\pi \colon S \to S/I$ denotes the quotient map.
Using that $\pi$ preserves suprema of increasing sequences, we obtain $x' \in S$ such that $x' \ll x$ and $\pi(x) \leq \pi(x')$, which then implies $\pi(x)=\pi(x')$.
By \cref{dfn:soft}, we obtain $n\geq 1$ and $t_1,\ldots,t_n\in S$ such that
\[
x'+t_j\ll x \quad \text{ for } j=1,\ldots,n, \andSep x'\ll t_1+\ldots+t_n.
\]

For each $j$, we get $\pi(x)+\pi(t_j)=\pi(x)$, and therefore
\[
2\pi(x) 
= \pi(x)+\pi(x')
\leq \pi(x) + \pi(t_1) + \ldots + \pi(t_n)
= \pi(x),
\]
as required.

It was shown in \cite[Proposition~5.3.5]{AntPerThi18TensorProdCu} that functional softness implies weak pure noncompactness and that the converse implication holds if $S$ satisfies \axiomO{5}.

Lastly, assume that $S$ satisfies \axiomO{5} and is residually stably finite, and that $x$ is functionally soft.
In this case, it was shown in \cite[Lemma~5.3.8]{AntPerThi18TensorProdCu} that~$x$ is strongly soft.
Note that a stronger notion of stable finiteness is used in \cite{AntPerThi18TensorProdCu}, but for the argument in \cite[Lemma~5.3.8]{AntPerThi18TensorProdCu} the weaker notion defined above suffices.
For convenience, we include the short argument.

To verify that $x$ is strongly soft, let $x' \in S$ satisfy $x' \ll x$.
Choose $y',y \in S$ such that $x' \ll y' \ll y \ll x$.

Since $x$ is functionally soft, we obtain $n\in\NN$ such that $(n+1)y \ll nx$.
Applying \axiomO{5} for $y' \ll y \leq x$, we obtain $t\in S$ such that
\[
y'+t \leq x \leq y+t.
\]

Let $I := \{ s\in S : s \leq \infty t \}$ be the ideal generated by $t$, and let $\pi \colon S \to S/I$ denote the quotient map.
Then $\pi(x)\leq\pi(y)$, and therefore
\[
n\pi(x)+\pi(x)
=(n+1)\pi(x)
\leq (n+1)\pi(y)
\ll n\pi(x).
\]

Using that $S/I$ is stably finite, we get $\pi(x)=0$ and, consequently, $x\leq\infty t$.
Using that $x' \ll x \leq \infty t$, we can choose $t' \in S$ such that $t' \ll t$ and $x' \ll \infty t'$.
Since $x' \ll y'$ and $t' \ll t$, we have $x'+t' \ll y'+t \leq x$, which shows that $t'$ has the desired properties.
\end{proof}

Recall that a \CuSgp{} is said to satisfy \emph{weak cancellation} if for all elements $x,y,z$ with $x+z \ll y+z$ we have $x \ll y$.
It is easy to check that weak cancellation passes to quotients and implies stably finiteness.
Hence, every weakly cancellative \CuSgp{} is residually stably finite.

It was shown in \cite[Theorem~4.3]{RorWin10ZRevisited} that the Cuntz semigroups of stable rank one \ca{s} always satisfy weak cancellation.

\begin{cor}
Let $A$ be a residually stably finite \ca{}.
(For example, a \ca{} with stable rank one.)
Then for elements in $\Cu(A)$, all notions from \cref{dfn:soft,dfn:pnc} are equivalent.
\end{cor}

\begin{prp}
\label{prp:CompactSoft}
Let $S$ be a \CuSgp{}, and let $x \in S$ be compact.
Then $x$ is strongly soft if and only if $x$ is weakly soft, if and only if $x$ is purely noncompact, if and only if $2x=x$.

Further, $x$ is functionally soft if and only if $x$ is weakly purely noncompact, if and only if there exists $n\in\NN$ such that $(n+1)x=nx$.
\end{prp}
\begin{proof}
If $2x=x$, then $x$ is easily seen to be strongly soft.
In general, strong softness implies weak softness, which in turn implies pure noncompactness, by \cref{prp:soft}.
Further, if $x$ is purely noncompact and compact, then $2x=x$ by definition.

Similarly, if $(n+1)x=nx$ for some $n$, then $x$ is easily seen to be functionally soft, which in turn implies that $x$ is weakly purely noncompact by \cref{prp:soft}.
Further, if $x$ is weakly purely noncompact and compact, then $(n+1)x=nx$ for some $n$ by definition.
\end{proof}

\begin{exa}
\label{exa:elementary}
Given $k\geq 1$, let $E_k=\{0,1,2,\ldots,k,\infty\}$, where the order is algebraic and the sum of two elements $x,y$ is  $x+y$ if $x+y\leq k$ or $\infty$ otherwise.
Then $E_k$ is a \CuSgp{} that satisfies \axiomO{5};
see \cite[Paragraph~5.1.16]{AntPerThi18TensorProdCu} for more details.

The element $x:=1$ is compact and satisfies $(k+2)x=\infty=(k+1)x$, but not $2x=x$.
Therefore, $x$ is functionally soft and weakly purely noncompact, but it is neither strongly soft, nor weakly soft, nor purely noncompact.
\end{exa}

\begin{exa}
\label{exa:Roerdam}
R{\o}rdam showed in \cite{Ror03FinInfProj} that there exist unital \ca{s} $A$ such that the unit of $A$ is finite, while the unit of $M_2(A)$ is properly infinite.
This implies that in $\Cu(A)$ the class $[1_A]$ of the unit of $A$ is a compact element satisfying $3[1_A] = 2[1_A] \neq [1_A]$.
By \cref{prp:CompactSoft}, $[1_A]$ is functionally soft and weakly purely noncompact, but it is neither strongly soft, nor weakly soft, nor purely noncompact.
\end{exa}

\cref{exa:Roerdam} shows that the lower downwards implications in \cref{prp:soft} cannot be reversed for elements in Cuntz semigroups.
The situation for the other implications is unclear.

\begin{qst}
For elements in the Cuntz semigroup of a \ca, does weak softness imply strong softness?
Does pure noncompactness imply weak softness?
\end{qst}



\begin{exa}
Let $S$ be a \CuSgp.
An element $x \in S$ is \emph{idempotent} if $2x=x$.
One says that $S$ is \emph{idempotent} if each of its elements is.
It is easy to see that idempotent elements are strongly soft.
Thus, every element in an idempotent \CuSgp{} is strongly soft.

Given a \ca{} $A$, it was noted in \cite[Section~7.2]{AntPerThi18TensorProdCu} that $\Cu(A)$ is idempotent if (and only if) $A$ is purely infinite.
Thus, every element in the Cuntz semigroup of a purely infinite \ca{s} is strongly soft.

A \ca{} $A$ is said to be \emph{weakly purely infinite} if there exists $n\in\NN$ such that $(n+1)x=nx$ for every $x\in\Cu(A)$. 
Such elements are functionally soft.
Thus, Cuntz semigroups of weakly purely infinite \ca{s} satisfy that every element is functionally soft.

It is an open problem if every weakly purely infinite \ca{} is purely infinite;
see \cite[Question~9.5]{KirRor02InfNonSimpleCalgAbsOInfty}.
\end{exa}

The next result provides a method of building strongly soft elements in \CuSgp{s}. 

\begin{lma}
\label{prp:ExaSoft}
Let $S$ be a \CuSgp{}, and let $(y_n)_n$ be a sequence in $S$ such that $y_n\leq \infty y_{n+1}$ for each $n\geq 0$. 
Then the element $\sum_{n=0}^\infty y_n$ is strongly soft.
\end{lma}
\begin{proof}
To verify that $y:=\sum_{n=0}^\infty y_n$ is strongly soft, let $y' \in S$ be such that $y'\ll y$. 
It follows that there exists some $m\in\NN$ such that $y'\ll \sum_{n=0}^m y_n$.
 
Using that $y_n\leq \infty y_{n+1}$ for each $n$, we obtain
\[
\sum_{n=0}^m y_n
\leq \infty y_1 + \sum_{n=1}^m y_n
\leq \infty y_2 + \sum_{n=2}^m y_n
\leq \ldots
\leq \infty y_m
\leq \infty y_{m+1}
\]
and, therefore, $y'\ll \infty y_{m+1}$. 

Let $t\in S$ be such that $t\ll y_{m+1}$ and $y'\ll \infty t$. 
Then,
\[
y'+t 
\ll \left(\sum_{n=0}^m y_n\right) + y_{m+1} 
\leq y, \andSep
y' \ll \infty t,
\]
as desired.
\end{proof}

\begin{prp}
\label{prp:CharStrongSoft}
Let $S$ be a \CuSgp, and let $x \in S$.
Then the following are equivalent:
\begin{enumerate}
\item
The element $x$ is strongly soft.
\item 
For every $x'\in S$ satisfying $x'\ll x$ there exists a strongly soft element $t\in S$ such that
\[
x'+t\leq x \leq\infty t.
\]
\item
For every $x'\in S$ satisfying $x'\ll x$ there exists $t\in S$ such that
\[
x'+t\leq x, \andSep x'\leq\infty t.
\]
\end{enumerate}
\end{prp}
\begin{proof}
It is clear that~(2) implies~(3).
To show that~(1) implies~(2), assume that~$x$ is strongly soft. 
Choose a $\ll$-increasing sequence $(x_n)_{n}$ with supremum $x$ and such that $x_0=x'$.
We inductively find $y_n$ and $t_n$ such that
\[
y_n+t_n\ll y_{n+1}\ll x, \quad
y_n \ll \infty t_n, \andSep 
x_{n+1}\ll y_{n+1}
\]
for all $n\geq 0$.
We start by setting $y_0:=x_0$.
Since $x$ is strongly soft, we obtain $t_0\in S$ such that $y_0+t_0\ll x$ and $y_0\ll\infty t_0$.
Now let $n\geq 0$ and assume that we have chosen $y_k$ and $t_k$ for all $k\leq n$.
Then, using that $y_n+t_n\ll x$ and $x_{n+1}\ll x$, we choose $y_{n+1}\ll x$ such that
\[
y_n+t_n\ll y_{n+1}, \andSep x_{n+1}\ll y_{n+1}.
\]

Since $x$ is strongly soft, we obtain $t_{n+1}\in S$ such that $y_{n+1}+t_{n+1}\ll x$ and $y_{n+1}\ll\infty t_{n+1}$, which finishes the induction argument.

For each $n\geq 0$, we have
\[
x_0+\sum_{k=0}^n t_k
= y_0+t_0+t_1+\ldots+t_n
\leq y_1+t_1+\ldots+t_n
\leq \ldots \leq y_{n+1}
\leq x.
\]

Setting $t:=\sum_{k=0}^\infty t_k$, it is clear that $x'+t\leq x\leq \infty t$. Further, it follows from \cref{prp:ExaSoft} that $t$ is strongly soft.

To show that~(3) implies~(1), assume that $x$ satisfies~(3), and let $x' \in S$ satisfy $x' \ll x$.
Choose $x'' \in S$ such that $x' \ll x'' \ll x$.
Applying~(3) to $x'' \ll x$, we obtain $t \in S$ such that
\[
x''+t \leq x, \andSep x'' \leq \infty t.
\]

Using that $x' \ll x''$, we can choose $t' \in S$ such that
\[
x '\ll \infty t', \andSep
t' \ll t,
\]
which implies that $x'+t'\ll x$, as desired.
\end{proof}

The next result is the analog of \cite[Theorem~5.3.11]{AntPerThi18TensorProdCu} for strongly soft elements.

\begin{thm}
\label{thm:SoftSubSgp}
Let $S$ be a \CuSgp, and let $S_\soft$ denote the set of strongly soft elements in $S$.
Then:
\begin{enumerate}
\item
The set $S_\soft$ is a submonoid of $S$ that is closed under suprema of increasing sequences.
\item
$S_\soft$ is absorbing in the following sense:
if $x,y\in S$ satisfy $x\leq\infty y$ and $y$ is strongly soft, then so is $x+y$.
\end{enumerate}
\end{thm}
\begin{proof}
To show that $S_\soft$ is closed under addition, let $x,y\in S_\soft$.
To verify that $x+y$ is strongly soft, let $w\in S$ satisfy $w\ll x+y$.
Choose $x',y'\in S$ such that
\[
w\ll x'+y', \quad x'\ll x, \andSep y'\ll y.
\]

Since $x$ and $y$ are strongly soft, we obtain $r,s\in S$ such that
\[
x'+r\ll x, \quad
x'\ll\infty r, \quad
y'+s\ll y, \andSep
y'\ll\infty s.
\]

Then
\[
w+(r+s) \leq x'+y'+r+s \ll x+y, \andSep
w\leq x'+y' \ll \infty (r+s).
\]

Using that $0$ is strongly soft, it follows that $S_\soft$ is a submonoid.

To see that $S_\soft$ is closed under suprema of increasing sequences, let $(x_n)_n$ be an increasing sequence in $S_\soft$, and set $x:=\sup_nx_n$.
To verify that $x$ is strongly soft, let $x'\in S$ satisfy $x'\ll x$.
We obtain $n\in\NN$ such that $x'\ll x_{n}$.
Since $x_{n}$ is strongly soft, we get $r\in S$ such that
\[
x'+r\ll x_{n}, \andSep
x'\ll\infty r.
\]

Then $x'+r\ll x$ and $x'\ll\infty r$, as desired.

Let us now prove~(2). Thus, let $x,y\in S$ satisfy $x\leq\infty y$, and assume that $y$ is strongly soft.
We verify condition~(3) of \cref{prp:CharStrongSoft} for $x+y$. Let $w\in S$ satisfy $w\ll x+y$, and choose $y'\in S$ such that
\[
w\ll x+y', \andSep
y'\ll y.
\]

Since $y$ is strongly soft, we can apply \cref{prp:CharStrongSoft} to obtain $r\in S$ such that
\[
y'+r\leq y\leq\infty r.
\]

Using that $x\leq\infty y$ at the third step, we get
\[
w+r\leq x+y'+r\leq x+y 
\leq \infty y \leq \infty r,
\]
as desired.
\end{proof}

The next result is the $\Cu$-version of \cref{prp:OrderZeroMapPresSoft}.

\begin{lma}
\label{prp:GenCuMor}
Let $\varphi \colon S \to T$ be a generalized \CuMor{} between \CuSgp{s}, and let $x \in S$.
If $x$ is (functionally, weakly, strongly) soft, then so is $\varphi(x)$.
\end{lma}
\begin{proof}
Set $y := \varphi(x)$.
Assuming that $x$ is strongly soft, we apply \cref{prp:CharStrongSoft}(3) to verify that $y$ is soft.
Let $y \in T$ satisfy $y' \ll y$.
Using that $\varphi$ preserves suprema of increasing sequences, we obtain $x' \in S$ such that $x' \ll x$ and $y' \leq \varphi(x')$.
By \cref{prp:CharStrongSoft}, we obtain $t \in S$ such that $x'+t \leq x \leq \infty t$.
Then
\[
y'+\varphi(t)
\leq \varphi(x')+\varphi(t)
= \varphi(x'+t)
\leq \varphi(x)
= y
\leq \varphi(\infty t)
= \infty \varphi(t),
\]
which shows that $\varphi(t)$ has the desired properties.

Similarly, one shows that functional and weak softness pass from $x$ to $y$.
\end{proof}

A projection $p$ in a \ca{} $A$ is said to be \emph{properly infinite} if $p\oplus p \precsim p\oplus 0$ in $M_2(A)$. 

Note that, whenever $A$ is residually stably finite, \cref{prp:SoftCaVsCu} below says that $a\in A_+$ is soft if and only if no Cuntz representative of $a+I$ is a nonzero projection; see \cite[Theorem~3.5]{BroCiu09IsoHilbModSF}. Thus, under these assumptions, one can understand softness as a generalization of pure positivity, as defined in \cite[Definition~2.1]{PerTom07Recasting}.

\begin{prp}
\label{prp:SoftCaVsCu}
Let $A$ be a \ca{}, and let $a\in A_+$.
Consider the following statements:
\begin{enumerate}
\item
The element $a$ is soft.
\item
The class $[a] \in \Cu (A)$ is strongly soft.
\item
For every closed ideal $I \subseteq A$, either $a \in I$, or the spectrum $\spec(a+I)$ has~$0$ as a limit point, or the characteristic function $\chi$ of $(0,\infty)$ is continuous on~$\spec(a)$ and the corresponding projection $p:=\chi(a)$ is properly infinite;
\end{enumerate}
Then the implications `(1)$\Rightarrow$(2)$\Rightarrow$(3)' hold.
If $A$ is residually stably finite, then the implication `(3)$\Rightarrow$(1)' holds and so~(1)-(3) are equivalent in this case.
\end{prp}
\begin{proof}
To show that~(1) implies~(2), set $x := [a] \in \Cu(A)$, and let $x' \in \Cu(A)$ satisfy $x' \ll x$. 
Then, there exists $\varepsilon > 0$ satisfying $x' \leq [(a-\varepsilon )_+] \leq x$.
Using \cref{prp:CharSoftElement}, we obtain $b\in \overline{aAa}_+$, orthogonal to $(a-\varepsilon )_+$, such that $a\in \overline{\linSpan} AbA$. 
Setting $t:=[b]$ in $\Cu (A)$, one gets 
\[
x'+t \leq x \leq \infty t.
\]
Now it follows from \cref{prp:CharStrongSoft} that $x=[a]$ is strongly soft.

To show that~(2) implies~(3), let $I \subseteq A$ be a closed ideal.
Then $\Cu(I)$ is naturally an ideal in $\Cu(A)$ such that $\Cu(A/I)$ is canonically isomorphic to $\Cu(A)/\Cu(I)$;
see \cite{CiuRobSan10CuIdealsQuot} and \cite[Section~5.1]{AntPerThi18TensorProdCu}.
The class $[a+I] \in \Cu(A/I)$ corresponds to the image of $[a]$ under the quotient map $\Cu(A) \to \Cu(A)/\Cu(I)$.
By \cref{prp:GenCuMor}, $[a+I]$ is strongly soft.

Thus, without loss of generality, we may assume that $a$ is nonzero and $I = \{0\}$.
We may also assume that $0$ is isolated in the spectrum $\sigma(a)$, which implies that there exists $\varepsilon>0$ with $\sigma(a) \subseteq \{0\} \cup [\varepsilon,\infty)$.
This implies that $\chi$ is continuous on $\sigma(a)$.
Applying functional calculus, we obtain a projection $p := \chi(a)$, which satisfies $[p]=[a]$.
We have $[p] \ll [p]$, and therefore $[p]=2[p]$ by \cref{prp:CompactSoft}.
This implies that $p\oplus p$ is Cuntz subequivalent, and thus also Murray-von Neumann subequivalent, to $p$. 
Hence, $p$ is a properly infinite projection.

Finally, if $A$ is residually stably finite, then no quotient of $A$ contains a nonzero, properly infinite projection, which shows that~(3) implies~(1) by \cref{prp:CharSoftElement}.
\end{proof}

\begin{qst}
\label{qst:RealizeSoftClass}
Let $A$ be a stable \ca, and let $x \in \Cu(A)$ be strongly soft.
Does there exist a soft element $a \in A_+$ with $x = [a]$?
\end{qst}

\section{Abundance of soft elements}
\label{sec:AbSoftEl}

In this section, we study \ca{s} and \CuSgp{s} with an \emph{abundance of soft elements}, a notion introduced in \cref{dfn:AbSoft} below.
We have shown in \cref{prp:SoftCaVsCu} that soft elements in \ca{s} have a (strongly) soft Cuntz class, and this readily implies that if a stable \ca{} has an abundance of soft elements, then so does the associated Cuntz semigroup.
The main result of this section shows that the converse also holds; 
see \cref{prp:EquivAbundance}.

For nonstable \ca{s}, we have to consider scaled \CuSgp{s}.
A \emph{scale} in a \CuSgp{} $S$ is a subset $\Sigma \subseteq S$ that is closed under suprema of increasing sequences, that is downward-hereditary, and that generates~$S$ as an ideal;
see \cite[Definition~4.1]{AntPerThi20CuntzUltraproducts}.
In this case, the pair $(S,\Sigma)$ is called a \emph{scaled \CuSgp}.

Given a \ca{} $A$, the naturally associated scale $\Sigma_A \subseteq \Cu(A)$ is
\[
\Sigma_A = \big\{ x \in \Cu(A) : x \leq [a] \text{ for some } a \in A_+ \big\};
\]
see \cite[4.2]{AntPerThi20CuntzUltraproducts} and \cite[Lemma~3.3(2)]{ThiVil22arX:Glimm}.

\medskip

Given elements $a$ and $b$ in a \ca{}, recall that we write $a \lhd b$ if $a$ belongs to the closed ideal generated by $b$. 
We will use analogous notation in \CuSgp{s}.

\begin{ntn}
Given elements $x$ and $y$ in a \CuSgp{} $S$, we write $x \lhd y$ if $x$ belongs to the ideal of $S$ generated by $y$, that is, if $x \leq \infty y$.
\end{ntn}

\begin{dfn}
\label{dfn:AbSoft}
We say that a \ca{} $A$ has \emph{an abundance of soft elements} if for every $a\in A_+$ and every $\varepsilon >0$ there exists a positive, soft element $b\in\overline{aAa}$ such that $(a-\varepsilon )_+ \lhd b$.

We say that a scaled \CuSgp{} $(S,\Sigma)$ has \emph{an abundance of strongly soft elements} if for every $x',x \in \Sigma$ satisfying $x' \ll x$ there exists a strongly soft element $y \in S$ such that $x' \lhd y \leq x$.
\end{dfn}

\begin{rmk}
Let $S$ be a \CuSgp{}, which we consider as a scaled \CuSgp{} with trivial scale $\Sigma = S$.
In this case, \cref{dfn:AbSoft} says that $S$ has \emph{an abundance of strongly soft elements} if for every $x',x \in S$ satisfying $x' \ll x$ there exists a strongly soft element $y \in S$ such that $x' \lhd y \leq x$.
\end{rmk}

We start with basic properties of the relation $\lhd$ for \CuSgp{s}.

\begin{lma}
\label{prp:BasicLhdCu}
Let $S$ be a \CuSgp, and let $x',x,y \in S$ satisfy $x' \ll x \lhd y$.
Then:
\begin{enumerate}
\item
There exists $y' \in S$ such that $x' \lhd y' \ll y$.
\item
If $S$ satisfies \axiomO{6} and \axiomO{7}, then there exists $z \in S$ such that $z \leq y$ and $x' \lhd z \lhd x$.
\end{enumerate}
\end{lma}
\begin{proof}
Statement~(1) is clear.
To verify~(2), assume that $S$ satisfies \axiomO{6} and \axiomO{7}.
Pick $x'' \in S$ such that $x' \ll x'' \ll x$.
Then $x'' \ll x \leq \infty y$, and we obtain $n \in \NN$ such that $x'' \leq ny$.

It follows from \axiomO{6}, applied for $x' \ll x'' \leq ny = y+\ldots+y$, that there exist $e_1,\ldots,e_n \in S$ such that
\[
x' \ll e_1 + \ldots + e_n, \andSep
e_j \leq x'',y \text{ for $j=1,\ldots,n$}.
\]

Choose $e_j' \in S$ such that
\[
x' \ll e_1' + \ldots + e_n', \andSep
e_j' \ll e_j \text{ for $j=1,\ldots,n$}.
\]

It follows from \axiomO{7}, applied for $e_j' \ll e_j \leq y$, that there exists $z \in S$ such that
\[
e_1',\ldots,e_n' \leq z \leq y, e_1+\ldots+e_n.
\]

Then
\[
x' 
\ll e_1' + \ldots + e_n'
\leq nz, \andSep
z 
\leq e_1+\ldots+e_n
\leq nx'',
\]
and therefore $x' \lhd z \lhd x$.
\end{proof}

Next, we show a useful consequence of having an abundance of soft elements in a \CuSgp;
see \cref{prp:SsoftsubCu}.
In preparation, we first prove the following lemma.

\begin{lma}
\label{prp:SsoftSubCuPre}
Let $(S,\Sigma)$ be a scaled \CuSgp{} with an abundance of strongly soft elements.
Then for every strongly soft element $x \in \Sigma$ and every $x' \in S$ with $x' \ll x$, there exists a strongly soft element $y \in S$ with $x' \ll y \ll x$.
\end{lma}
\begin{proof}
We adapt the proof of \cite[Proposition~5.3.18]{AntPerThi18TensorProdCu}. 
Choose $z',z \in S$ with $x' \ll z' \ll z \ll x$.
Since $x$ is strongly soft, we know by \cref{prp:CharStrongSoft} that there exists $t \in S$ such that
\[
z+t \leq x \leq \infty t.
\]
Take $t' \in S$ such that
\[
t' \ll t, \andSep 
z \ll \infty t'.
\]
 
Since $t \leq x \in \Sigma$, we have $t \in \Sigma$.
Using that $(S,\Sigma)$ has an abundance of soft elements, we obtain a strongly soft element $u \in S$ such that
\[
t' \lhd u \ll t.
\]

Set $y := z'+u$.
Then
\[
x'
\ll z' 
\leq z'+u
= y, \andSep
y 
= z'+u
\ll z+t
\leq x.
\]
Further, using that $z' \leq \infty t' \leq \infty u$, it follows from \cref{thm:SoftSubSgp} that $y = z'+u$ is strongly soft. 
\end{proof}

As defined in \cite[Definition 4.1]{ThiVil21DimCu2}, recall that a submonoid $T$ of a \CuSgp{} $S$ is a \emph{sub-\CuSgp{}} of $S$ if it is closed under suprema of increasing sequences and is such that, for every $x\in T$ and $x'\in S$ with $x'\ll x$, there exists $y\in T$ such that $x'\ll y\ll x$.

\begin{prp}
\label{prp:SsoftsubCu}
Let $S$ be a \CuSgp{} with an abundance of strongly soft elements.
Then the set $S_\soft$ of strongly soft elements in $S$ forms a sub-\CuSgp{} of $S$.
\end{prp}
\begin{proof}
By \cref{thm:SoftSubSgp}, $S_\soft$ is a submonoid of $S$ that is closed under suprema of increasing sequences.
Further, by \cref{prp:SsoftSubCuPre}, for all $x \in S_\soft$ and $x' \in S$ with $x'\ll x$, there exists $y \in S_\soft$ such that $x' \ll y \ll x$.
\end{proof}

In \cite[Definition~5.1]{KirRor15CentralSeqCharacters}, Kirchberg and R{\o}rdam define a \ca{} $A$ to have the \emph{$2$-splitting property} if there exist full elements $a,b\in A_+$ with $a \perp b$.
They only apply their definition to unital \ca{s}.
We propose below a notion that considers the $2$-splitting property for all hereditary sub-\ca{s}, and that also allows for a small `error' (which is only relevant for nonunital subalgebras).

\begin{dfn}
\label{dfn:Her2Splitting}
We say that a \ca{} $A$ has the \emph{hereditary $2$-splitting property} if for every $a\in A_+$ and every $\varepsilon >0$ there exist positive elements $b,c \in \overline{aAa}$ such that $b \perp c$ and $(a-\varepsilon )_+ \lhd b,c$.

We say that a scaled \CuSgp{} $(S,\Sigma)$ has the \emph{hereditary $2$-splitting property} if for every $x',x \in \Sigma$ satisfying $x' \ll x$ there exist $y,z \in S$ such that $y+z \leq x$, and $x' \lhd y,z$.
\end{dfn}

We first show that in \CuSgp{s} satisfying \axiomO{5}-\axiomO{7} (which always hold in Cuntz semigroups of \ca{s}), one can ensure that in the definition of the hereditary $2$-splitting property at least one of the `splitting' elements generates an ideal that not only contains $x'$ but even $x$.

\begin{lma}
\label{prp:PreCuEquiv}
Let $(S,\Sigma)$ be a scaled \CuSgp{} satisfying \axiomO{5}-\axiomO{7} and with the hereditary $2$-splitting property.
Let $x', x \in \Sigma$ be such that $x' \ll x$. 
Then there exist $y,z \in \Sigma$ such that 
\[
y+z \leq x, \quad
x' \lhd y, \andSep
x \lhd z.
\]
\end{lma}
\begin{proof}
Choose $x_1,x_2,x_3 \in S$ such that
\[
x' \ll x_1 \ll x_2 \ll x_3 \ll x.
\]

Using the hereditary $2$-splitting property for $x_3 \ll x$, we obtain $s,t \in S$ such that
\[
s+t \leq x, \andSep
x_3 \lhd s,t.
\]

Applying \cref{prp:BasicLhdCu}(2) for $x_1 \ll x_2 \lhd s$, we obtain $s'$ such that
\[
x_1 \lhd s' \lhd x_2, \andSep
s' \leq s.
\]

Applying \cref{prp:BasicLhdCu}(1) for $x' \ll x_1 \lhd s'$ and for $x_2 \ll x_3 \lhd t$, we obtain $s''$ and $t'$ such that
\[
x' \lhd s'' \ll s', \andSep
x_2 \lhd t' \ll t.
\]

Using \axiomO{5} for $s'+t \leq x$, and $s'' \ll s'$, and $t' \ll t$, we obtain $c \in S$ such that
\[
s''+c \leq x \leq s'+c, \andSep
t' \leq c.
\]
We have $s' \lhd x_2 \lhd t' \leq c$, and therefore $x \lhd c$.
Further, $x' \lhd s''$.
This shows that $y := s''$ and $z := c$ have the desired properties.
\end{proof}

\begin{thm}
\label{prp:CuEquivAbSoft}
Let $(S,\Sigma)$ be a scaled \CuSgp{} satisfying \axiomO{5}-\axiomO{7}. 
Then the following are equivalent:
\begin{enumerate}
\item 
For every $x \in \Sigma$ there exists a strongly soft $y \in S$ such that $x \lhd y \leq x$.
\item 
$(S,\Sigma)$ has an abundance of strongly soft elements, that is, for every $x',x \in \Sigma$ with $x' \ll x$, there exists a strongly soft $y \in S$ with $x' \lhd y \leq x$.
\item
$(S,\Sigma)$ has the $2$-splitting property, that is, for every $x',x \in \Sigma$ with $x' \ll x$, there exists $y,z \in S$ with $y+z \leq x$ and $x' \lhd y,z$.
\end{enumerate}
\end{thm}
\begin{proof}
It is clear that~(1) implies~(2).
Using \cref{prp:CharStrongSoft}, we see that~(2) implies~(3).

To prove that~(3) implies~(1), let $x \in \Sigma$.
We need to find a strongly soft $y \in S$ with $y \leq x \lhd y$.

Choose a $\ll$-increasing sequence $(x_n)_n$ with supremum $x$.
We will inductively find elements $z_n \in S$ for $n \geq 0$ and $y_n \in S$ for $n \geq 1$ such that 
\[
y_{n+1} + z_{n+1} \leq z_{n}, \quad 
x_{n+1} \lhd y_{n+1}, \andSep
x \lhd z_n
\]
for all $n\in\NN$.

To begin, set $z_0 := x$. 
Now, let $n \geq 0$, and assume that we have chosen $z_n$.
Using \cref{prp:BasicLhdCu}(1) for $x_{n+1} \ll x \lhd z_n$, we obtain $z_n'$ such that
\[
x_{n+1} \lhd z_n' \ll z_n.
\]

Applying \cref{prp:PreCuEquiv} for $z_n' \ll z_n$, we obtain $y_{n+1},z_{n+1}$ such that
\[
y_{n+1}+z_{n+1} \leq z_n , \quad
z_n' \lhd y_{n+1}, \andSep
z_n \lhd z_{n+1}.
\]

Then $x_{n+1} \lhd z_n' \lhd y_{n+1}$ and $x \lhd z_n \lhd z_{n+1}$, as desired.

Next, we alter the elements $y_n$ so that they generate an increasing sequence of ideals.
Given $n \geq 1$, apply \cref{prp:BasicLhdCu}(2) for $x_{n-1} \ll x_n \lhd y_n$ to obtain $y_n'$ such that
\[
x_{n-1} \lhd y_n' \lhd x_n, \andSep
y_n' \leq y_n.
\]

Set $y := \sum_{n=1}^\infty y_n'$.
We get $y_n' \lhd x_n \lhd y_{n+1}'$ for every $n\geq 1$.
Therefore, $y$ is strongly soft by \cref{prp:ExaSoft}. 

For each $n \geq 1$, we have
\begin{align*}
y_1'+\ldots+y_{n-1}'+y_n'
&\leq y_1+\ldots+y_{n-1}+y_n+z_n \\
&\leq y_1+\ldots+y_{n-1}+z_{n-1}
\leq \ldots 
\leq y_1+z_1
\leq z_0 = x
\end{align*}
and therefore $y = \sup_n \sum_{j=1}^n y_j' \leq x$.
Further, for every $n \geq 1$, we have
\[
x_n \lhd y_n' \leq y,
\]
and, consequently, $x = \sup_n x_n \lhd y$.
\end{proof}

\begin{rmk}
In \cref{prp:CuEquivAbSoft}, the implications `(1)$\Rightarrow$(2)$\Rightarrow$(3)' hold for arbitrary \CuSgp{s}.
The assumptions \axiomO{5}-\axiomO{7} are only needed for the the implication `(1)$\Leftarrow$(3)'.
\end{rmk}

The following lemma is the $C^*$-analogue of \cref{prp:ExaSoft}.

\begin{lma}
\label{prp:ConstructSoft}
Let $A$ be a \ca{}, and let $(a_n)_n$ be a sequence of positive, contractive, pairwise orthogonal elements in $A$.
Assume that $a_n \lhd a_{n+1}$ for every $n\in\NN$.
Then $a:=\sum_{n=0}^\infty \tfrac{1}{2^n}a_n$ is soft.
\end{lma}
\begin{proof}
We verify condition~(3) of \cref{prp:CharSoftElement}. 
Let $\varepsilon>0$.
Choose $N\in\NN$ such that $\tfrac{1}{2^N}\leq\varepsilon$.
For $n\geq N$, we have $\tfrac{1}{2^n}a_n-\varepsilon\leq 0$, and therefore $(\tfrac{1}{2^n}a_n-\varepsilon)_+=0$.
Since the $a_n$'s are pairwise orthogonal, we have
\[
(a-\varepsilon)_+ 
= \sum_{n=0}^\infty (\tfrac{1}{2^n}a_n - \varepsilon)_+
= \sum_{n=0}^N (\tfrac{1}{2^n}a_n - \varepsilon)_+.
\]

Set $b:=\sum_{n=N+1}^\infty \tfrac{1}{2^n}a_n$.
Then $b\perp(a-\varepsilon)_+$.
For $n\geq N+1$, it is clear that $a_n\in \overline{\linSpan}AbA$. 
For $n\leq N$, it follows from the assumption that $a_n\in \overline{\linSpan}AbA$.
This implies that $a\in \overline{\linSpan}AbA$, as desired.
\end{proof}

The following lemma is the $C^*$-analogue of \cref{prp:BasicLhdCu}.

\begin{lma}
\label{prp:BasicLhd}
Let $A$ be a \ca{}, and let $a,b\in A_+$ satisfy $a \lhd b$.
Then:
\begin{enumerate}
\item
For every $\varepsilon>0$, there exists $\delta>0$ such that $(a-\varepsilon)_+\lhd(b-\delta)_+$.
\item
For every $\varepsilon>0$, there exists $c\in\overline{bAb}_+$ such that $(a-\varepsilon)_+\lhd c \lhd a$.
\end{enumerate}
\end{lma}
\begin{proof}
To prove~(1), take $\varepsilon>0$. 
Using \cite[Corollary~II.5.2.13]{Bla06OpAlgs}, let $N\in\NN$ and $r_n\in A$ for $n=1,\ldots ,N$ be such that $\Vert a - \sum_{n=1}^N r_n^* br_n\Vert <\varepsilon $. 
Then, there exists $\delta >0$ such that 
\[
\left\| a - \sum_{n=1}^N r_n^* (b-\delta )_+r_n \right\| <\varepsilon 
\]

Applying \cite[Lema~2.5(ii)]{KirRor00PureInf}, we obtain an element $r\in A$ such that $(a-\varepsilon)_+=\sum_{n=1}^N (r_n r)^* (b-\delta )_+ (r_n r)$, as required.

To prove~(2), take $\varepsilon>0$.
Set $I := \overline{\linSpan} AaA$.
Then $\overline{bAb} \cap I$ generates $I$ as a closed ideal.
Thus, using an approximate unit in $\overline{bAb} \cap I$, we can find a positive element $c \in \overline{bAb} \cap I$ such that $a$ is at distance at most $\varepsilon$ from the ideal generated by $c$.
Using a similar argument as in~(1), one can show that $(a-\varepsilon)_+ \lhd c$.
Thus, $c$ has the desired properties.
\end{proof}

\begin{lma}
\label{prp:SoftAlgCu}
Let $A$ be a \ca{}, let $a,c \in A_+$, let $\varepsilon>0$, and let $(s_n)_n$ be a sequence of strongly soft elements in $\Cu(A)$ such that $[a] \leq \infty s_n$ for all $n$, and 
\[
\ldots \ll s_2 \ll s_1 \ll s_0 \leq [c].
\]
Then there exists a soft element $b \in \overline{cAc}$ such that $(a-\varepsilon )_+ \lhd b$.  
\end{lma}
\begin{proof}
For each $n\in\NN$, using that $s_{n+1} \ll s_n$ and that $s_n$ is strongly soft, we obtain $t_{n+1} \in \Cu(A)$ such that
\[
s_{n+1} + t_{n+1} \ll s_n, \andSep
s_{n+1} \ll \infty t_{n+1}.
\]

Choose $s_n',t_{n+1}' \in \Cu(A)$ such that
\[
s_n' \ll s_n, \quad
s_{n+1} + t_{n+1} \ll s_n', \quad
t_{n+1}' \ll t_{n+1}, \andSep
s_{n+1} \ll \infty t_{n+1}'.
\]

Next, we inductively choose $c_n \in A_+$ for $n \in \NN$ and $d_n \in A_+$ for $n \geq 1$ such that $c_{n+1},d_{n+1}$ are orthogonal elements in $\overline{c_nAc_n}$, and such that 
\[
s_n' \leq [c_n] \leq s_n,\andSep 
t_n' \leq [d_n] \leq t_n.
\]

For $n=0$, set $c_0:=c$.
Now fix $n\in\NN$ and assume that we have already chosen the elements $c_0,\ldots,c_n \in A_+$ and $d_1,\ldots, d_n \in A_+$.
Then
\[
s_{n+1}' \ll s_{n+1}, \quad
t_{n+1}' \ll t_{n+1}, \andSep
s_{n+1} + t_{n+1} \leq s_n' \leq [c_n],
\]

Applying \cite[Lemma~2.3(ii)]{RobRor13Divisibility}, we find orthogonal elements $c_{n+1}$ and $d_{n+1}$ in~$\overline{c_n Ac_n}$ such that 
\[
s_{n+1}' \ll [c_{n+1}] \ll s_{n+1}, \andSep 
t_{n+1}' \ll [d_{n+1}] \ll t_{n+1},
\]
which finishes the inductive argument.

Now choose a strictly decreasing sequence $(\gamma_n)_n$ of positive numbers such that
\[
\ldots < \gamma_2 < \gamma_1 < \gamma_0=\varepsilon.
\]

Given $n \geq 1$, we have
\[
[a] 
\leq \infty s_n
\leq \infty t_n'
\leq \infty [d_n],
\]
and therefore $(a-\gamma_{n+1})_+ \lhd a \lhd d_n$.

Using the equality $(a-\gamma_n)_+ = ((a-\gamma_{n+1})_+-(\gamma_n-\gamma_{n+1}))_+$, and by applying \cref{prp:BasicLhd}~(2), we obtain a positive element $b_n$ in $\overline{d_nAd_n}$ such that
\[
(a-\gamma_n)_+ \lhd b_n \lhd (a-\gamma_{n+1})_+.
\]

By construction, the elements $b_1,b_2,\ldots$ are pairwise orthogonal, and $b_n \lhd b_{n+1}$ for each $n \geq 1$.
Thus, the element $b:=\sum_{n=1}^\infty \tfrac{1}{2^k\|b_k\|}b_k$ is soft by \cref{prp:ConstructSoft}. Note that, since each $b_n$ is in $\overline{cAc}$, we have $b\in\overline{cAc}$. Moreover, one gets
\[
(a-\varepsilon)_+ = (a-\gamma_0)_+ \lhd b_1 \lhd b,
\]
as desired.
\end{proof}

\begin{thm}
\label{prp:EquivAbundance}
Let $A$ be a \ca.
Then the following are equivalent:
\begin{enumerate}
\item
$A$ has an abundance of soft elements. 
\item
$A$ has the hereditary $2$-splitting property. 
\item
$(\Cu(A),\Sigma_A)$ has an abundance of strongly soft elements.
\item
$(\Cu(A),\Sigma_A)$ has the hereditary $2$-splitting property.
\end{enumerate}
\end{thm}
\begin{proof}
To show that~(1) implies~(2), assume that $A$ has an abundance of soft elements, and let $a\in A_+$ and $\varepsilon>0$.
We need to find positive elements $b,c \in \overline{aAa}$ such that $b \perp c$ and $(a-\varepsilon )_+ \lhd b,c$.
By assumption, there exists a soft element $d \in \overline{aAa}$ with $(a-\tfrac{\varepsilon}{2})_+ \lhd d$.
Note that
\[
(a-\varepsilon)_+
= \left( (a-\tfrac{\varepsilon}{2})_+ - \tfrac{\varepsilon}{2} \right)_+.
\]

We can therefore apply \cref{prp:BasicLhd}(1) to obtain $\delta>0$ such that
\[
(a-\varepsilon)_+ \lhd (d-\delta)_+.
\]

Using that $d$ is soft, we get from \cref{prp:CharSoftElement}(3) a positive element $c \in \overline{dAd}$ with $c \perp (d-\delta)_+$ and $d \lhd c$.
Set $b := (d-\delta)_+$.
Then $b$ and $c$ have the desired properties.

Let us show that~(2) implies~(4).
Let $x',x \in \Sigma_A$ with $x' \ll x$.
We need to find $y,z \in \Cu(A)$ with $y+z \leq x$ and $x' \lhd y,z$.

By \cite[Lemma~3.3(1)]{ThiVil22arX:Glimm}, we obtain $a \in A_+$ such that $x' \ll [a] \ll x$.
Choose $\varepsilon>0$ such that $x' \ll [(a-\varepsilon)_+]$.
By assumption, we obtain orthogonal positive elements $b,c \in \overline{aAa}$ with $(a-\varepsilon)_+ \lhd b,c$.
Then $y := [b]$ and $z := [c]$ have the desired properties.

Since Cuntz semigroups of \ca{s} satisfy \axiomO{5}-\axiomO{7}, it follows from \cref{prp:CuEquivAbSoft} that~(4) implies~(3).
Finally, to show that~(3) implies~(1), assume that $(\Cu(A),\Sigma_A)$ has an abundance of strongly soft elements and take $a \in A_+$ and $\varepsilon > 0$.
We need to find a soft element $b \in \overline{aAa}$ with $(a-\varepsilon)_+ \lhd b$.

Using that $(\Cu(A),\Sigma_A)$ has an abundance of strongly soft elements for $[a] \in \Sigma_A$ and $[(a- \tfrac{\varepsilon}{3})_+] \ll [a]$, we obtain a strongly soft $s_0 \in \Cu(A)$ such that
\[
[(a- \tfrac{\varepsilon}{3})_+] \lhd s_0 \leq [a].
\]
Since $[(a- \tfrac{2\varepsilon}{3})_+] \ll [(a- \tfrac{\varepsilon}{3})_+]$, we can choose $s_0' \in \Cu(A)$ such that
\[
[(a- \tfrac{2\varepsilon}{3})_+] \lhd s_0' \ll s_0.
\]
Applying \cref{prp:SsoftSubCuPre} for $s_0' \ll s_0$, we obtain a strongly soft $s_1 \in \Cu(A)$ with $s_0' \ll s_1 \ll s_0$.
Inductively, we obtain a sequence $(s_n)_n$ of strongly soft elements in $\Cu(A)$ such that 
\[
\ldots \ll s_2 \ll s_1 \ll s_0 \leq [a],
\]
and such that $[(a- \tfrac{2\varepsilon}{3})_+] \lhd s_0' \lhd s_n$ for all $n$.

Applying \cref{prp:SoftAlgCu} (with $\tfrac{\varepsilon}{3}$), we obtain a soft element $b \in \overline{aAa}$ such that 
\[
[(a-\varepsilon)_+]
= [((a- \tfrac{2\varepsilon}{3})_+-\tfrac{\varepsilon}{3})_+]
\lhd b.
\]
as desired.
\end{proof}

\begin{cor}
\label{prp:EquivAbundanceStable}
A stable \ca{} has an abundance of soft elements if and only if its Cuntz semigroup has an abuncande of strongly soft elements.
\end{cor}

Let $A$ be a (nonstable) \ca{}.
If $A\otimes\KK$ has an abundance of soft elements, then so does $A$.
Using that $\Cu(A)\cong\Cu(A\otimes\KK)$, we deduce that $A$ has an abundance of soft elements whenever its Cuntz semigroup does. 
The converse remains unclear:

\begin{qst}
Let $A$ be a \ca.
If a $A$ has an abundance of soft elements, does its stabilization as well?
\end{qst}

With view towards \cref{prp:EquivAbundance}, the above question is equivalent to the following:
If $A$ is a \ca{} such that $(\Cu(A),\Sigma_A)$ has an abundance of strongly soft elements, does $\Cu(A)$ have an abundance of strongly soft elements?
One can also ask this question in the more abstract setting of scaled \CuSgp{s} that satisfy additional properties, like \axiomO{5}-\axiomO{8}.

\section{Completely soft elements in C*-algebras}

We introduce completely soft operators as those positive elements in a \ca{} whose every cut-down is soft;
see \cref{dfn:CompSoft}.
We characterize these elements by spectral properties;
see \cref{prp:CharCompSoft}.
The main result of this section is \cref{prp:EquivAbundanceSoft}, where we show that a \ca{} has an abundance of such elements whenever it has an abundance of soft elements.
To construct completely soft elements in a \ca{}, we follow a strategy introduced in \cite{Thi20arX:diffuseHaar} and view positive elements as suitable `paths' of open projections.

\begin{dfn}
\label{dfn:CompSoft}
We say that a positive element $a$ in a \ca{} is \emph{completely soft} if for every $\varepsilon \geq 0$ the element $(a-\varepsilon )_+$ is soft.
\end{dfn}

\begin{prp}
\label{prp:CharCompSoft}
Let $A$ be a \ca{}, and $a \in A_+$.
Then the following are equivalent:
\begin{enumerate}
\item
The element $a$ is completely soft.
\item
For every closed ideal $I \subseteq A$, the spectrum of $a+I \in A/I$ is $[0,\|a+I\|]$.
\item
For every closed ideal $I \subseteq A$, the spectrum of $a+I \in a/I$ is connected and contains $0$.
\end{enumerate}
\end{prp}
\begin{proof}
Let $\varepsilon \geq 0$, let $I \subseteq A$ be a closed ideal, and let $\pi_I \colon A \to A/I$ denote the quotient map.
Then $(\pi_I(a)-\varepsilon)_+ = \pi_I((a-\varepsilon)_+)$.
Using the spectral mapping theorem (see, for example, \cite[Proposition~II.2.3.2]{Bla06OpAlgs}), we see that $\varepsilon$ is contained in the closure of $\spec(\pi_I(a))\cap(\varepsilon,\infty)$ if and only if $0$ is not isolated in the spectrum of $(\pi_I(a)-\varepsilon)_+$.

By \cref{prp:CharSoftElement}, $a$ is completely soft if and only if for every $\varepsilon \geq 0$ and every closed ideal $I \subseteq A$, the element $\pi_I((a-\varepsilon)_+)$ is either zero or $0$ is not isolated in its spectrum.
Using the above considerations, this is in turn equivalent to~(2), which is easily seen to also be equivalent to~(3).
\end{proof}

An immediate consequence of \cref{prp:CharCompSoft} is that images of completely soft elements in quotients are again completely soft, which is analogous to \cref{prp:SoftImageInQuotient}.
It remains unclear if the analog of \cref{prp:OrderZeroMapPresSoft} also holds:

\begin{qst}
Let $\varphi \colon A \to B$ be a completely positive, order-zero map between \ca{s}.
Given a completely soft element $a \in A_+$, is $\varphi(a)$ completely soft?
\end{qst}

\begin{pgr}
\label{pgr:OpenProj}
Let $A$ be a \ca{}. 
Recall that a projection $p\in A^{**}$ is said to be \emph{open} if it is the weak*-limit of an increasing net in $A_+$, and denote by $\OP(A)$ the set of all open projections in $A$. 
The sub-\ca{} $pA^{**}p\cap A$ is hereditary for every open projection $p$ and, conversely, every hereditary sub-\ca{} of $A$ is of this form for some unique open projection of $A$; 
see \cite[p.77f]{Ped79CAlgsAutGp}.
 
As in \cite[Section~2]{Thi20arX:diffuseHaar}, we define the relation $\prec$ on $\OP(A)$ by setting $p\prec q$ if there exists $a\in A_+$ such that $p\leq a\leq q$. 
That is, if there exists $a\in A_+$ such that $p=pa$ and $a=aq$. 
We say that $p\in\OP(A)$ is \emph{soft} if the associated hereditary sub-\ca{} $A_p:=pA^{**}p\cap A$ is soft. 
We use $\supp (a)$ to denote the support projection of any positive element $a\in A_+$.
Note that $\supp(a)$ belongs to $\OP(A)$.

Following \cite[Definition~2.4]{Thi20arX:diffuseHaar}, we say that a map $f \colon (-\infty,0] \to \OP(A)$ is a \emph{path} if $f(s) = \sup\{f(s') : s' < s \}$ for all $s \in (-\infty,0]$, and if $f(s) \prec f(t)$ whenever $s<t$ in $(-\infty,0]$.

Given a positive contraction $a \in A$, the map $f_a \colon (-\infty,0] \to \OP(A)$ given by $f_a(t) := \supp((a+t)_+)$ is a path with $f_a(-1)=0$.
Conversely, by \cite[Proposition~2.6]{Thi20arX:diffuseHaar}, every path $f$ with $f(-1)=0$ arises this way from a positive contraction.
\end{pgr}

\begin{lma}
\label{prp:InterpolateSoft}
Let $A$ be a \ca{} that has an abundance of soft elements, let $p,q\in\OP(A)$ satisfy $p \prec q$, and assume that $q$ is soft.
Then there exists $r\in\OP(A)$ such that $r$ is soft and $p \prec r \prec q$.
\end{lma}
\begin{proof}
Given any $\varepsilon >0$, let $f_\varepsilon\colon [0,\infty )\to [0,1]$ be the map that takes the value $0$ in $[0,\varepsilon /2]$, $1$ in $[\varepsilon ,\infty )$, and is linear from $\varepsilon /2$ to $\varepsilon$.

Since $p\prec q$, there exists $a\in A_+$ such that $p\leq a\leq q$. Thus, one has 
\[
 p\leq f_1 (a)\ll f_{1/2} (a)\ll f_{1/4}(a)\leq q.
\]

Using that $f_{1/4}(a)\in A_q$, it follows that there exists an element in $(A_q)_+$ that acts as a unit for $f_{1/2} (a)$.
By \cref{prp:CharSoftAlg}, we obtain a positive element $d\in A_q$ with $f_{1/2} (a)\lhd d\perp f_{1/2} (a)$.
Applying \cref{prp:BasicLhd}(1), we obtain $\delta>0$ such that $(f_{1/2} (a)-\tfrac{1}{2})_+ \lhd (d-\delta)_+$. Note that one has $f_1(a)\lhd (f_{1/2} (a)-\tfrac{1}{2})_+$.

Applying that $A$ has an abundance of soft elements for $(d-\tfrac{\delta}{2})_+$ and $\tfrac{\delta}{2}$, we obtain a soft element $e$ such that
\[
\big( (d-\tfrac{\delta}{2})_+ - \tfrac{\delta}{2} \big)_+ \lhd e, \andSep
e \in \overline{(d-\tfrac{\delta}{2})_+ A (d-\tfrac{\delta}{2})_+},
\] 
and we may assume that $e$ is contractive.

Then, one has
\[
f_1 (a)
\lhd (f_{1/2} (a)-\tfrac{1}{2})_+
\lhd (d-\delta)_+
= \big( (d-\tfrac{\delta}{2})_+ - \tfrac{\delta}{2} \big)_+ 
\lhd e.
\]

Further, using that $f_1 (a)\perp e$ and $f_1 (a)\lhd e$, it follows that $f_1 (a)+e$ is soft and, consequently, that the open projection $r:=\supp(f_1 (a)+e)$ is also soft.
Since $f_1 (a)+e$ is contractive, we have $p\leq f_1 (a) \leq f_1 (a)+e\leq\supp(f_1 (a)+e)=r$ and thus $p \prec r$.

We also have $e \in \overline{(d-\tfrac{\delta}{2})_+ A (d-\tfrac{\delta}{2})_+}$ and $(d-\tfrac{\delta}{2})_+ \ll f_{\delta/2}(d)$, and therefore $e \ll f_{\delta/2}(d)$.
We further have $f_1 (a) \ll f_{1/2} (a)$, and $f_{1/2} (a)\perp f_{\delta/2}(d)$.
This implies that
\[
f_1 (a)+e \ll f_{1/2} (a)+f_{\delta/2}(d).
\]

Therefore,
\[
r=\supp(f_1 (a)+e) \leq f_{1/2} (a)+f_{\delta/2}(d)\leq q,
\]
and thus $r\prec q$.
\end{proof}

In preparation for the proof of \cref{prp:WeakInterpolation}, we recall a result of Bice and Koszmider from \cite{BicKos19LLApproxUnits}.
Following \cite[Deﬁnition~2.1]{BicKos19LLApproxUnits}, we write $a \ll_\varepsilon b$ for positive elements $a$ and $b$ in a \ca{} and $\varepsilon>0$ if $\|a-ab\| < \varepsilon$.

\begin{thm}[{\cite[Corollary~4.2]{BicKos19LLApproxUnits}}]
\label{prp:BicKos}
For all $\varepsilon>0$ there exists $\delta>0$ such that the following holds:
If $a,b,c,d$ are positive contractions in a \ca{} $A$ such that
\[
a \ll b \ll_\delta c \ll d
\]
then there exists a unitary $u \in \widetilde{A}$ such that
\[
uau^* \ll d, \andSep
\| u - 1 \| < \varepsilon.
\]
\end{thm}

Given open projections $p,q,r \in \OP(A)$ such that $p,q \prec r$, note that generally there is no open projection $r'$ satisfying $p,q \prec r' \prec r$. 
However, by using results from \cite{BicKos19LLApproxUnits}, we  prove in \cref{prp:WeakInterpolation} below that an approximate version of this statement does hold.

We write $p \prec_\varepsilon q$ if there exists $a\in A_+$ such that $\| p - pa \| < \varepsilon$ and $a\leq q$.

\begin{lma}
\label{prp:WeakInterpolation}
Let $p,q,r$ be open projections in a \ca{} $A$ such that $p,q\prec r$.
Take $\varepsilon>0$. 
Then there exists $r'\in\OP (A)$ such that $p \prec r' \prec  r$ and $q \prec_\varepsilon r'$.
\end{lma}
\begin{proof}
Using functional calculus, we find positive, contractive elements $a',a,b',b$ in~$A$ such that 
\[
p \leq a' \ll a \leq r, \andSep
q \leq b' \ll b \leq r.
\]

Fix $\delta>0$ such that the statement in \cref{prp:BicKos} is satisfied for $\varepsilon/3$ and $\delta$, where note that we may assume $\delta\leq\varepsilon/3$. Using functional calculus and an approximate unit in $A_r$, we find positive, contractive elements $c',c \in A$ such that
\[
a \ll_\delta c' \ll c, \andSep
b \ll_\delta c' \ll c.
\]

Now, by \cref{prp:BicKos}, there exists a unitary $u \in \widetilde{A_r}$ such that
\[
ua'u^* \ll c, \andSep
\| 1 - u \| < \frac{\varepsilon}{3}.
\]

Then
\[
a' \ll u^*cu, \andSep
b \ll_\delta c' \ll_{2\varepsilon/3} u^*cu,
\]
and therefore $b \ll_\varepsilon u^*cu$.
Now $r' := \supp(u^*cu)$ has the desired properties.
\end{proof}

\begin{lma}
\label{prp:SupSoftOp}
Let $A$ be a \ca{} and let $(p_j)_j$ be a $\prec$-directed family of soft open projections in $A$. 
Then the open projection $\sup_j p_j$ is soft.
\end{lma}
\begin{proof}
Let $p$ be the supremum of $(p_j)_j$. Then, 
\[
 A\cap pA^{**}p =\overline{A\cap \left(\bigcup_j p_j A^{**}p_j\right)}.
\]

It follows that each $A\cap p_j A^{**}p_j$ is a sub-\ca{} of $A\cap pA^{**}p$, and that the family $(A\cap p_j A^{**}p_j)_j$ approximates $A\cap pA^{**}p$. 
Since each $A\cap p_j A^{**}p_j$ is soft, we know from \cref{prp:PermApprox} that $A\cap pA^{**}p$ is soft. 
By definition, $p$ is soft.
\end{proof}

\begin{thm}
\label{prp:AbCompSoft}
Let $A$ be a \ca{} that has an abundance of soft elements, and let $a\in A_+$ be soft.
Then there exists a contractive, completely soft element $b$ such that $\overline{aAa}=\overline{bAb}$.
\end{thm}
\begin{proof}
Set $p(0) := \supp(a)$. Let us first prove that there exists a sequence of soft open projections $p(-\tfrac{1}{2^n})\in\OP (A)$ such that
\[
\supp((a-\tfrac{1}{2^n})_+) \prec_{1/2^n} p(-\tfrac{1}{2^n})\prec p(-\tfrac{1}{2^{n+1}})\prec p(0)
\]
for each $n\geq 1$.

For $n=1$, it follows from \cref{prp:InterpolateSoft} applied for $\supp((a-\tfrac{1}{2})_+)\prec p(0)$ that one can find a soft open projection $p(-\tfrac{1}{2})$ with the required properties.

Now fix $n\in\NN$ and assume that we have found the projections $p(-\tfrac{1}{2^k})$ for each $k\leq n$. Applying \cref{prp:WeakInterpolation} for $p(-\tfrac{1}{2^n}),\supp((a-\tfrac{1}{2^{n+1}})_+) \prec p(0)$ and $\varepsilon=\tfrac{1}{2^{n+1}}$ we obtain $r\in\OP(A)$ such that
\[
p(-\tfrac{1}{2^n}) \prec r \prec p(0), \andSep
\supp((a-\tfrac{1}{2^{n+1}})_+) \prec_{1/2^{n+1}} r.
\]

Since $r\prec p(0)$, it follows from  \cref{prp:InterpolateSoft} that there exists a soft open projection $p(-\tfrac{1}{2^{n+1}})$ such that $r \prec p(-\tfrac{1}{2^{n+1}}) \prec p(0)$, as desired. In particular, note that $p(0)=\sup_n p(-\tfrac{1}{2^n})$.

Repeated application of the previous construction to each $p(-\tfrac{1}{2^n})$ allows us to find soft open projections $p(t)$ for every dyadic number $t\in[-1,0]$ such that $p(-1)=0$, and such that $p(s) \prec p(t)$ whenever $s<t$.

We set $r(-1)=0$ and 
\[
r(t) := \sup \big\{ p(s) : s \in [-1,t)  \text{ dyadic} \big\}.
\]
for each $t\in(-1,0]$.

Note that 
\[
r(0) \leq p(0), \andSep
r(0) \geq \sup_n p(-\tfrac{1}{2^n}) = p(0),
\]
which implies $r(0)=p(0)$. Further, each $r(t)$ is soft by \cref{prp:SupSoftOp}.

It follows that $r\colon[-1,0]\to\OP(A)$ is a path in the sense of \cite[Definition~2.4]{Thi20arX:diffuseHaar}. By \cite[Proposition~2.6]{Thi20arX:diffuseHaar}, there is a unique positive element $b\in A_+$ such that $\supp((b+t)_+)=r(t)$ for each $t\in[-1,0]$.
Since $r(-1)=0$, the element $b$ is contractive.
Further, since $\supp(b)=r(0)=p(0)=\supp(a)$, we have $\overline{aAa}=\overline{bAb}$.

Finally, each cut-down of $b$ is soft by construction, as desired.
\end{proof}

Recall from \cref{dfn:AbSoft} that we say that a \ca{} $A$ has an \emph{abundance of soft elements} if for every $a \in A_+$ and every $\varepsilon >0$ there exists a positive, soft element $b\in\overline{aAa}$ such that $(a-\varepsilon )_+ \lhd b$.
One could therefore phrase statement~(2) in the next result by saying that `$A$ has an abundance of completely soft elements'.

\begin{thm}
\label{prp:EquivAbundanceSoft}
Let $A$ be a \ca.
Then the following are equivalent:
\begin{enumerate}
\item
The \ca{} $A$ has an abundance of soft elements.
\item
For every $a\in A_+$ and $\varepsilon>0$ there exists a positive, completely soft element $b\in\overline{aAa}$ such that $(a-\varepsilon)_+ \lhd b$.
\end{enumerate}
\end{thm}
\begin{proof}
It is clear that~(2) implies~(1).
Conversely, assume that~$A$ has an abundance of soft elements.
To verify~(2), take $a\in A_+$ and $\varepsilon>0$. 
Since $A$ has an abundance of soft elements, one can find a soft element $c\in\overline{aAa}_+$ such that $(a-\varepsilon)_+ \in\overline{\linSpan}AcA$. 
By \cref{prp:AbCompSoft}, there exists a compleltely soft element $b\in A_+$ such that $\overline{cAc}=\overline{bAb}$. This implies that $b\in\overline{aAa}_+$ and $(a-\varepsilon)_+ \in\overline{\linSpan}AbA$, which shows~(2).
\end{proof}

\section{Softness and the Global Glimm Property}
\label{sec:Glimm}

Following \cite[Definition~4.12]{KirRor02InfNonSimpleCalgAbsOInfty}, we say that a \ca{} $A$ has the \emph{Global Glimm Property} if for every $a\in A_+$, every $\varepsilon >0$ and every $k\geq 2$ there exists a $\ast$-homomorphism $\varphi\colon C_0 (0,1]\otimes M_k\to \overline{aAa}$ such that $(a-\varepsilon)_+$ is in the ideal generated by the image of $\varphi$.

Every \ca{} with the Global Glimm Property has no nonzero elementary ideal-quotients (that is, it is \emph{nowhere scattered} in the sense of \cite{ThiVil21arX:NowhereScattered}). 
It is an open problem, known as the \emph{Global Glimm Problem}, to determine if every nowhere scattered \ca{} has the Global Glimm Property; 
see \cite{ThiVil22arX:Glimm}.

We shed new light on this problem by proving that the Global Glimm Property implies having an abundance of soft elements, which in turn implies nowhere scatteredness;
see \cref{prp:AbSoftImplNWS,prp:2DivGivesStronglySoft}.
This decomposes the Global Glimm Problem into two subproblems;
see \cref{qst:ReverseImplications}.

\medskip

To prove these results, we use that nowhere scatteredness and the Global Glimm Property of a \ca{} $A$ are reflected in divisibility properties (recalled in \cref{pgr:Divisibility}) of its Cuntz semigroup $\Cu(A)$.
More precisely, $A$ is nowhere scattered if and only if $\Cu(A)$ is weakly $(2,\omega)$-divisible;
see \cite[Theorem~8.9]{ThiVil21arX:NowhereScattered}.
Further, $A$ has the Global Glimm Property if and only if $\Cu(A)$ is $(2,\omega)$-divisible;
see \cite[Theorem~3.6]{ThiVil22arX:Glimm}.

This translates the Global Glimm Problem into a question about Cuntz semigroups:
Does weak $(2,\omega)$-divisibility imply $(2,\omega)$-divisibility?

In \cite{ThiVil22arX:Glimm} we identified two conditions on a \CuSgp{} that precisely capture when the desired implication of divisibility properties holds: ideal-fitlererdness and property~(V).
More precisely, a \CuSgp{} satisfying \axiomO{5}-\axiomO{8} is $(2,\omega)$-divisible if and only if it is weakly $(2,\omega)$-divisible, ideal-filtered and has property (V); 
see \cite[Theorem~6.3]{ThiVil22arX:Glimm}.
Below, we give straightforward generalizations of ideal-filteredness and property~(V) to scaled \CuSgp{s} (\cref{dfn:IFScaled,dfn:VScaled}), and we note that the main result of \cite{ThiVil22arX:Glimm} can be generalized to the scaled setting.

We show that an abundance of strongly soft elements implies property~(V) (\cref{prp:AbSoftImplV}), which implies that a \CuSgp{} satisfying \axiomO{5}-\axiomO{8} is $(2,\omega)$-divisible if and only if it is ideal-filtered and has an abundance of strongly soft elements;
see \cref{prp:CharDivWithSoft}.
It follows that a \ca{} has the Global Glimm Property if and only if it has an abundance of soft elements and its scaled Cuntz semigroup is ideal-filtered; 
see \cref{prp:CharGlimm}.

\begin{pgr}
\label{pgr:Divisibility}
As defined in \cite[Definition~5.1]{RobRor13Divisibility}, a \CuSgp{} $S$ is said to be \emph{weakly $(2,\omega )$-divisible} if, for every pair $x'\ll x$ in $S$, there exist finitely many elements $y_1,\ldots ,y_n\in S$ such that $x'\leq y_1+\ldots +y_n$ and $2y_j\leq x$ for each $j\leq n$. 
 
Similarly, one says that $S$ is \emph{$(2,\omega )$-divisible} if, whenever $x'\ll x$ in $S$, then there exists $y\in S$ and $n\in\NN$ such that $2y\leq x$ and $x'\leq ny$.
\end{pgr}

In \cite[Lemma~3.5]{ThiVil22arX:Glimm}, we showed that a \CuSgp{} satisfying \axiomO{5}-\axiomO{7} is $(2,\omega)$-divisible whenever all elements in a scale are.
The next result shows that the analog statement holds for weak $(2,\omega)$-divisibility.

\begin{lma}
\label{prp:WeaklyDivScale}
Let $(S,\Sigma)$ be a scaled \CuSgp{} satisfying \axiomO{6} and assume that every element in $\Sigma$ is weakly $(2,\omega)$-divisible.
Then $S$ is weakly $(2,\omega)$-divisible.
\end{lma}
\begin{proof}
To show that every element in $S$ is weakly $(2,\omega)$-divisible, let $x',x \in S$ satisfy $x' \ll x$.
Pick $x'' \in S$ with $x' \ll x'' \ll x$.
Since $\Sigma$ is a scale, there exist $n$ and $y_1,\ldots,y_n \in \Sigma$ such that $x'' \leq y_1+\ldots+y_n$.
Applying \axiomO{6}, we obtain $z_1,\ldots,z_n$ such that
\[
x' \ll z_1 + \ldots + z_n, \andSep
z_j \leq x \text{ for $j=1,\ldots,n$}.
\]

Pick $z_j'$ such that
\[
x' \ll z_1' + \ldots + z_n', \andSep
z_j' \ll z_j \text{ for $j=1,\ldots,n$}.
\]

Each $z_j$ belongs to $\Sigma$ and is therefore weakly $(2,\omega)$-divisible.
We thus obtain elements $v_{j,1},\ldots,v_{j,N_j}$ such that
\[
z_j' \leq v_{j,1} + \ldots + v_{j,N_j}, \andSep
2v_{j,k} \leq z_j \text{ for $k=1,\ldots,N_j$}.
\]

Then $2v_{j,k} \leq z_j \leq x$ for each $j$ and $k$, and also
\[
x' \leq \sum_j z_j' \leq \sum_{j,k} v_{j,k}.
\]

This shows that the elements $v_{j,k}$ have the desired properties.
\end{proof}

Next, we prove that a \ca{} is nowhere scattered whenever it has an abundance of strongly soft elements; 
see \cref{prp:AbSoftImplNWS} below.
In preparation for the proof, we need the following lemma.

\begin{lma}
\label{prp:WeaklySoftWeaklyDiv}
Let $S$ be a \CuSgp{} satisfying \axiomO{6}, and let $x\in S$ be weakly soft.
Then $x$ is weakly $(2,\omega)$-divisible.
\end{lma}
\begin{proof}
To verify that $x$ is weakly $(2,\omega)$-divisible, let $x'\in S$ satisfy $x'\ll x$.
Choose $x''\in S$ such that $x'\ll x''\ll x$.
Since $x$ is weakly soft, we obtain $n\geq 1$ and $t_1,\ldots,t_n\in S$ such that
\[
x''+t_1\ll x, \quad\ldots, \quad 
x''+t_n\ll x, \andSep 
x''\ll t_1+\ldots+t_n.
\]

Applying \axiomO{6} to $x'\ll x''\ll t_1+\ldots+t_n$, we obtain $z_1,\ldots,z_n\in S$ such that
\[
x'\ll z_1+\ldots+z_n, \quad
z_1\ll x'',t_1, \quad\ldots, \quad
z_n\ll x'',t_n.
\]

Then
\[
2z_j \ll x''+t_j \ll x
\]
for each $j=1,\ldots,n$.
Thus, $z_1,\ldots,z_n$ have the desired properties.
\end{proof}

\begin{prp}
\label{prp:AbSoftImplNWS}
Let $(S,\Sigma)$ be a scaled \CuSgp{} satisfying \axiomO{6} that has an abundance of strongly soft elements.
Then $S$ is weakly $(2,\omega)$-divisible.

Consequently, if a \ca{} has an abundance of soft elements, then it is nowhere scattered.
\end{prp}
\begin{proof}
By \cref{prp:WeaklyDivScale}, it suffices to verify that every element in $\Sigma$ is weakly $(2,\omega)$-divisible.
Let $x',x \in \Sigma$ satisfy $x' \ll x$.
Pick $x'',x''' \in S$ with $x' \ll x'' \ll x'''\ll x$.
Then there exists a strongly soft element $y \in S$ satisfying $x''' \lhd y \ll x$.
Take $y' \in S$ with $y' \ll y$ and $x'' \lhd y'$.
By \cref{prp:WeaklySoftWeaklyDiv}, $y$ is weakly $(2,\omega)$-divisible, which means that we can find $n$ and $z_1,\ldots,z_n \in S$ such that
\[
y' \leq z_1+\ldots+z_n, \andSep
2z_j \leq y \text{ for $j=1,\ldots,n$}.
\]
Using that $x' \ll x'' \lhd y'$, we obtain $N\in\NN$ such that $x' \leq Ny'$.
Then the elements with the desired properties are given by considering $N$ copies of each of $z_1,\ldots,z_n$.

For the second part of the statement, let $A$ be a \ca{} with an abundance of soft elements.
By \cref{prp:EquivAbundance}, the scaled \CuSgp{} $(\Cu(A),\Sigma_A)$ has an abundance of strongly soft elements.
Since Cuntz semigroups of \ca{s} always satisfy \axiomO{6}, we can apply the above argument to deduce that $\Cu(A)$ is weakly $(2,\omega)$-divisible.
Hence, $A$ is nowhere scattered by \cite[Theorem~8.9]{ThiVil21arX:NowhereScattered}.
\end{proof}

\cref{prp:2DivGivesStronglySoft} below showcases the exact difference between having an abundance of strongly soft elements and being $(2,\omega )$-divisible: 
In the first case, for every element $x$ one can find a strongly soft element $y$ such that $y\leq x\leq \infty y$ (\cref{prp:CuEquivAbSoft}). 
In the second case, $y$ can be chosen such that $2y\leq x\leq \infty y$.

In the proof of \cref{prp:kDivSeq} below, we will apply the following useful consequence of \axiomO{5}:

\begin{lma}[{\cite[Lemma~2.2]{ThiVil22arX:Glimm}}]
\label{prp:ComplementOfDivisor}
Let $S$ be a \CuSgp{} satisfying \axiomO{5}, and let $k \in \NN$ and $x',x,z \in S$ satisfy $x' \ll x$ and $(k+1)x \leq z$.
Then there exists $y \in S$ such that $kx' + y \leq z \leq (k+1)y$ and $x' \ll y$.
\end{lma}

\begin{lma}
\label{prp:kDivSeq}
Let $k\in\NN$, and let $S$ be a $(2,\omega )$-divisible \CuSgp{} satisfying \axiomO{5}. 
Then, whenever $(x_n)_n$ is a $\ll$-increasing sequence in $S$, there exists a sequence $(y_n)_n$ such that 
\[
\sum_{n=0}^\infty ky_n \leq \sup_n x_n,\andSep 
y_n,x_{n+1}\ll \infty y_{n+1}
\]
for every $n\geq 0$.
\end{lma}
\begin{proof}
We will use below that $S$ is $(k+1,\omega)$-divisible by \cite[Lemma~3.4]{ThiVil22arX:Glimm}.

Let $(x_n)_n$ be a $\ll$-increasing sequence in $S$, and let $x$ be its supremum. 
Using induction on $n$, we will find elements $y_j, c_j \in S$ for $j\in\NN$ such that 
\[
ky_{j+1}+ c_{j+1} \leq c_{j}\leq \infty c_{j+1},\quad 
y_{j}\ll c_{j},\andSep 
y_{j},x_{j+1}\ll \infty y_{j+1}.
\]
for all $j\geq 0$, and with $ky_0+c_0\leq x\leq (k+1)c_0$.

First, applying that $S$ is $(k+1,\omega)$-divisible for $x_0 \ll x$, we find $z \in S$ such that 
\[
(k+1)z \ll x,\andSep 
x_0 \ll \infty z.
\]
 
Choose $y_0 \in S$ such that
\[
y_0 \ll z, \andSep x_0 \ll \infty y_0.
\]
Then, it follows from \cref{prp:ComplementOfDivisor} that there exists $c_0 \in S$ with 
\[
ky_0 + c_0\leq x\leq (k+1)c_0,\andSep y_0\ll c_0,
\]
and it is readily checked that $c_0,y_0$ satisfy the required conditions.

Now fix $n\in\NN$ and assume that we have found the elements $y_j,c_j$ for every $j\leq n$. 
In particular, we have $x\leq \infty c_0\leq \infty c_n$. 
Choose $c_n' \in S$ such that
\[
c_n' \ll c_n, \quad
y_n \ll c_n', \andSep 
x_{n+1} \ll \infty c_n'.
\] 
Applying that $S$ is $(k+1,\omega )$-divisible for $c_n' \ll c_n$, one finds $z_{n+1}\in S$ such that 
\[
(k+1)z_{n+1}\ll c_n,\andSep 
c_n'\ll \infty z_{n+1}.
\]
Choose $y_{n+1} \in S$ such that
\[
y_{n+1} \ll z_{n+1}, \andSep c_n'\ll \infty y_{n+1}.
\]
Applying \cref{prp:ComplementOfDivisor}, we find $c_{n+1} \in S$ with 
\[
ky_{n+1} + c_{n+1} \leq c_n \leq (k+1)c_{n+1}, \andSep 
y_{n+1} \ll c_{n+1},
\]
which finishes the inductive argument.

Finally, note that for every $n\in\NN$, one has 
\[
\begin{split}
ky_0+\ldots +ky_n 
&\leq ky_0+\ldots +ky_n + c_n
\leq ky_0+\ldots +ky_{n-1}+c_{n-1} \\
&\leq \ldots 
\leq ky_0+c_0\leq x.
\end{split}
\]
 
Taking supremum on $n$, we obtain the required result.
\end{proof}

\begin{prp}
\label{prp:2DivGivesStronglySoft}
Let $S$ be a $(2,\omega)$-divisible \CuSgp{} satisfying \axiomO{5}. 
Then for every $x\in S$ and $k \in \NN$ there exists a strongly soft element $y\in S$ such that $ky\leq x\leq\infty y$.
In particular, $S$ has an abundance of strongly soft elements.

Consequently, if a \ca{} has the Global Glimm Property, then it has an abundance of soft elements.
\end{prp}
\begin{proof}
Let $x\in S$ and $k\in\NN$, and let $(x_n)_n$ be a $\ll$-increasing sequence in $S$ with supremum $x$. Using \cref{prp:kDivSeq}, we find a sequence $(y_n)_n$ such that
\[
\sum_{n=0}^\infty ky_n \leq x,\andSep 
y_n,x_{n+1}\ll \infty y_{n+1}
\]
for every $n\geq 0$.

Set $y:=\sum_{n=0}^\infty y_n$, which satisfies $ky\leq x$ by construction. Then, since we have
\[
x_n\leq\infty y_n \leq \infty y,
\]
for each $n$, we deduce $x\leq\infty y$.
Further, $y$ is strongly soft by \cref{prp:ExaSoft}.

For the second part of the statement, let $A$ be a \ca{} with the Global Glimm Property.
Then $\Cu(A)$ is $(2,\omega)$-divisible by \cite[Theorem~3.6]{ThiVil22arX:Glimm}. 
Since Cuntz semigroups of \ca{s} always satisfy \axiomO{5}, we can apply the above argument to deduce that $\Cu(A)$ has an abundance of strongly soft elements.
Hence, $A$ (and in fact, even $A\otimes\KK$) has an abundance of soft elements by \cref{prp:EquivAbundance}.
\end{proof}

If a \ca{} has the Global Glimm Property, then it has an abundance of soft elements by \cref{prp:2DivGivesStronglySoft}; 
and if a \ca{} has an abundance of soft elements, then it is nowhere scattered by \cref{prp:AbSoftImplNWS}.
Therefore, the Global Glimm Problem, which asks if every nowhere scattered \ca{} has the Global Glimm Property, decomposes into the following subquestions:

\begin{qst}
\label{qst:ReverseImplications}
If a \ca{} is nowhere scattered, does it have an abundance of soft elements?
If a \ca{} has an abundance of soft elements, does it have the Global Glimm Property?
\end{qst}

In \cite[Definitions~4.1,~5.1]{ThiVil22arX:Glimm}, we defined ideal-filteredness and property~(V) for \CuSgp{s}.
We will use the following straightforward generalizations to scaled \CuSgp{s}.

\begin{dfn}
\label{dfn:IFScaled}
We say that a scaled \CuSgp{} $(S,\Sigma)$ is \emph{ideal-filtered} if for all $v',v \in S$ and $x,y \in \Sigma$ satisfying
\[
v' \ll v \ll \infty x, \infty y,
\]
there exists $z \in S$ such that
\[
v' \ll \infty z, \andSep
z \ll x,y.
\]
\end{dfn}

\begin{dfn}
\label{dfn:VScaled}
We say that a scaled \CuSgp{} $(S,\Sigma)$ has \emph{property~(V)} if, given $d_1',d_1,d_2',d_2,c,x\in \Sigma$ such that 
\[
d_1' \ll d_1, \quad 
d_2' \ll d_2, \quad 
d_1,d_2 \ll c, \andSep 
c+d_1,c+d_2 \ll x,
\]
there exist $y,z\in S$ satisfying
\[
y+z \leq x, \andSep 
d_1'+d_2'\leq \infty y,\infty z.
\]
\end{dfn}

\begin{lma}
\label{prp:AbSoftImplV}
Let $(S,\Sigma)$ be a scaled \CuSgp{} with an abundance of strongly soft elements. 
Then $(S,\Sigma)$ has property~(V).
\end{lma}
\begin{proof}
Let $d_1',d_1,d_2',d_2,c,x \in \Sigma$ be such that 
\[
d_1' \ll d_1, \quad 
d_2' \ll d_2, \quad 
d_1,d_2 \ll c, \andSep 
c+d_1,c+d_2 \ll x.
\]

Choose $x'',x' \in S$ such that
\[
x'' \ll x' \ll x, \andSep 
c+d_1,c+d_2\ll x''.
\]

Using that $(S,\Sigma)$ has an abundance of strongly soft elements for $x' \ll x$, we find a strongly soft element $s \in S$ with $x' \lhd s \ll x$. 
Choose $y \in S$ such that $x'' \lhd y \ll s \ll x$.

Since $s$ is strongly soft, we know by \cref{prp:CharStrongSoft} that there exists $z\in S$ satisfying $y+z\leq s\leq \infty z$. 
It is readily checked that the elements $y,z$ have the desired properties.
\end{proof}

\begin{thm}
\label{prp:CharDivWithSoft}
Let $S$ be a \CuSgp{} satisfying \axiomO{5}-\axiomO{8}.
Then the following are equivalent:
\begin{enumerate}
\item
$S$ is $(2,\omega)$-divisible.
\item
Every element in $\Sigma$ is $(2,\omega)$-divisible.
\item
Every element in $\Sigma$ is weakly $(2,\omega)$-divisible and $(S,\Sigma)$ is ideal-filtered and has property~(V).
\item
$(S,\Sigma)$ is ideal-filtered and for every $x$ in $S$, there exists a strongly soft element $y\in S$ such that $2y\leq x\leq\infty y$;
\item
$(S,\Sigma)$ is ideal-filtered and has an abundance of strongly soft elements.
\end{enumerate}
\end{thm}
\begin{proof}
By \cite[Lemma~3.5]{ThiVil22arX:Glimm}, (1) and~(2) are equivalent.

The results and proofs of Section~6 in \cite{ThiVil22arX:Glimm} have straightforward generalizations to the scaled version of (weak) $(2,\omega)$-divisibility, ideal-filteredness, and property~(V).
In particular, the scaled version of \cite[Theorem~6.3]{ThiVil22arX:Glimm} shows that~(2) and~(3) are equivalent.

Using \cref{prp:2DivGivesStronglySoft}, we see that~(1) implies~(4), and it is clear that~(4) implies~(5).
Finally, by \cref{prp:AbSoftImplV} and \cref{prp:AbSoftImplNWS}, (5) implies~(3).
\end{proof}


\begin{cor}
Let $S$ be a \CuSgp{} satisfying \axiomO{5}-\axiomO{8}. Assume that $S=S_\soft$. 
Then~$S$ is $(2,\omega)$-divisible if and only if $S$ is ideal-filtered.
\end{cor}

\begin{thm}
\label{prp:CharGlimm}
Let $A$ be a \ca{}. Then the following are equivalent:
\begin{enumerate}
\item $A$ has the Global Glimm Property.
\item $A$ has and abundance of soft elements and $(\Cu (A),\Sigma_A )$ is ideal-filtered.
\item $A$ has and abundance of soft elements and $\Cu (A)$ is ideal-filtered.
\end{enumerate}
\end{thm}
\begin{proof}
Let $A$ be a \ca{}. 
It follows from \cite[Theorem~3.6]{ThiVil22arX:Glimm} that $A$ has the Global Glimm Property if and only if $\Cu(A)$ is $(2,\omega)$-divisible. Thus, \cite[Theorem~6.3]{ThiVil22arX:Glimm} together with \cref{prp:2DivGivesStronglySoft} show that (1)~implies~(3), and it is clear that (3)~implies~(2).

Using \cref{prp:CharDivWithSoft}, we know that $\Cu(A)$ is $(2,\omega)$-divisible if and only if $\Cu (A)$ has an abundance of strongly soft elements and $(\Cu (A),\Sigma_A )$ is ideal-filtered. By \cref{prp:EquivAbundance}, $A$ has an abundance of soft elements if and only if strongly soft elements are abundant in $(\Cu(A),\Sigma_A)$. This shows that (2)~implies~(1), as desired.
\end{proof}


\providecommand{\bysame}{\leavevmode\hbox to3em{\hrulefill}\thinspace}
\providecommand{\noopsort}[1]{}
\providecommand{\mr}[1]{\href{http://www.ams.org/mathscinet-getitem?mr=#1}{MR~#1}}
\providecommand{\zbl}[1]{\href{http://www.zentralblatt-math.org/zmath/en/search/?q=an:#1}{Zbl~#1}}
\providecommand{\jfm}[1]{\href{http://www.emis.de/cgi-bin/JFM-item?#1}{JFM~#1}}
\providecommand{\arxiv}[1]{\href{http://www.arxiv.org/abs/#1}{arXiv~#1}}
\providecommand{\doi}[1]{\url{http://dx.doi.org/#1}}
\providecommand{\MR}{\relax\ifhmode\unskip\space\fi MR }
\providecommand{\MRhref}[2]{%
  \href{http://www.ams.org/mathscinet-getitem?mr=#1}{#2}
}
\providecommand{\href}[2]{#2}

\end{document}